\documentclass[a4paper,12pt]{amsart}
\usepackage{etex}
\usepackage{tikz, tikz-cd}

\usepackage{bm}
\usepackage{amssymb}
\usepackage{mathrsfs}
\usepackage{amsmath, amssymb,amsthm,latexsym,amscd,mathrsfs}
\usepackage{indentfirst}
\usepackage{stmaryrd}
\usepackage{graphicx}
\usepackage{subfigure}
\usepackage{extarrows}
\usepackage{amssymb}
\usepackage{amsmath,amssymb,amsthm,latexsym,amscd,mathrsfs}
\usepackage{indentfirst}
\usepackage{multirow}
\usepackage{latexsym}
\usepackage{amsfonts}
\usepackage{color}
\usepackage{pictexwd,dcpic}
\usepackage{graphicx}
\usepackage{psfrag}
\usepackage{hyperref}
\usepackage{comment}
\usepackage{cases}

\makeatletter
\@namedef{subjclassname@2010}{{\rm2020} Mathematics Subject Classification}
\makeatother

 \newcommand{\R}{\mathbb{R}}
 \newcommand{\C}{\mathbb{C}}
 \newcommand{\ROM}[1]{\mathrm{\uppercase\expandafter{\romannumeral#1}}}
  \theoremstyle{definition}
   \numberwithin{equation}{section} \theoremstyle{plain}

 \newtheorem{thm}{Theorem}[section]
 \newtheorem{lem}{Lemma}[section]

 \newtheorem{rem}{Remark}[section]
 \newtheorem{prop}{Proposition}[section]

  \makeatletter

  \numberwithin{equation}{section}
  \allowdisplaybreaks
\makeatother
\title[Topology and curvature of isoparametric families in spheres]{\textbf{Topology and curvature of isoparametric families in spheres}}
\author[Chao Qian]{Chao Qian}\address{School of Mathematics and Statistics, Beijing Institute of Technology, Beijing
100081, P.R. China}
\email{6120150035@bit.edu.cn}
\author[Z. Z. Tang]{Zizhou Tang}\address{Chern Institute of Mathematics $\&$ LPMC, Nankai University, Tianjin 300071, P. R. China}
\email{zztang@nankai.edu.cn}
\author[W. J. Yan]{Wenjiao Yan}\address{School of Mathematical Sciences, Laboratory of Mathematics and Complex Systems, Beijing Normal University, Beijing, 100875, P. R. China}
\email{wjyan@bnu.edu.cn}
\thanks {The project is partially supported by the NSFC (No.11722101, 11871282, 11931007), BNSF (Z190003), Nankai Zhide Foundation and Beijing Institute of Technology Research Fund Program for Young Scholars.
}


\begin{document}

\maketitle

\begin{abstract}
An isoparametric family  in the unit sphere consists of parallel isoparametric hypersurfaces and their two focal submanifolds. The present paper has two parts. The first part investigates topology of the isoparametric families, namely the homotopy, homeomorphism, or diffeomorphism types, parallelizability, as well as the Lusternik-Schnirelmann category. This part extends substantially the results of Q.M.Wang in \cite{Wa88}. The second part is concerned with their curvatures, more precisely, we determine when they have non-negative sectional curvatures or positive Ricci curvatures with the induced metric.
\end{abstract}

\section{\textbf{Introduction}}

An \emph{isoparametric hypersurface}  $M^n$ in the unit sphere $S^{n+1}(1)$ is a hypersurface with constant principal curvatures. Such a hypersurface always occurs as part of a family of the level hypersurfaces of an \emph{isoparametric function} $f$, which is a smooth function on $S^{n+1}(1)$ such that
\begin{equation}\label{ab}
\left\{ \begin{array}{ll}
|\nabla f|^2= b(f),\\
\,\,~\triangle f~~=a(f),
\end{array}\right.
\end{equation}
where $\nabla f$ and $\triangle f$ are the gradient and Laplacian of $f$, respectively; $b$ is a smooth function on $\mathbb{R}$, and $a$ is a continuous function on $\mathbb{R}$.
As is well known, an isoparametric family consists of regular parallel hypersurfaces with constant mean curvatures and two singular level sets carrying manifold structure, which are called focal submanifolds.

Denote the number of distinct principal curvatures of a closed isoparametric hypersurface $M^n$ by $g$, and the principal curvatures by $\lambda_1>\lambda_2>\cdots>\lambda_g$ with multiplicities $m_1,\ldots, m_g$, respectively.
Recall a celebrated result of M{\"u}nzner that $g$ can be only $1, 2, 3, 4$ or $6$, $m_i=m_{i+2}$ (subscripts mod $g$), the principal curvatures could be written as $$\lambda_i=\cot(\theta+\frac{i-1}{g}\pi)$$
 with $\theta\in (0, \frac{\pi}{g})$ $(i=1,\ldots,g)$, and a closed, connected isoparametric hypersurface $M$ must be a level set of the restriction to $S^{n+1}(1)$ of a homogeneous polynomial $F: \mathbb{R}^{n+2} \rightarrow \mathbb{R}$ of degree $g$ satisfying the \emph{Cartan-M{\"u}nzner equations:}
\begin{equation}\label{ab}
\left\{ \begin{array}{ll}
|\nabla F|^2= g^2|x|^{2g-2}, \\
~~~~\triangle F~~=\frac{m_2-m_1}{2}g^2|x|^{g-2}.
\end{array}\right.
\end{equation}
Such a polynomial $F$ is called the \emph{Cartan-M{\"u}nzner polynomial}, and $f=F~|_{S^{n+1}(1)}$ takes values in $[-1, 1]$. For $-1 < t < 1$, $f^{-1}(t)$ is an isoparametric hypersurface. The level sets $M_+=f^{-1}(1)$ and  $M_-=f^{-1}(-1)$ are the two focal submanifolds with codimensions $m_1+1$ and $m_2+1$ in $S^{n+1}(1)$, which are proved to be minimal submanifolds of the unit sphere $S^{n+1}(1)$ (c.f. \cite{No73}, \cite{Mun80}, \cite{Mun81}, \cite{Wan87}, \cite{GT14}). It is noteworthy that Fang \cite{Fa17} gave a simplified proof for M\"{u}nzner's result on the restriction of $g$, based on the work of \cite{GH87}.

With much effort of many geometers through many years, the classification of the isoparametric hypersurfaces in $S^{n+1}(1)$ was accomplished recently, which settles the 34th problem of S.T.Yau's ``Open problems in Geometry". Clearly, when $g=1$, an isoparametric hypersurface in $S^{n+1}(1)$ is a hypersphere, and the focal submanifolds are just two points. When $g=2$, an isoparametric hypersurface in $S^{n+1}(1)$ is isometric to the generalized Clifford torus $S^p(r)\times S^q(s)$ with $r^2+s^2=1$ and $p+q=n$, and the focal submanifolds are isometric to $S^p(1)$ and $S^q(1)$, respectively. When $g=3$, E. Cartan showed that the multiplicities must obey $m_1=m_2=m=1,2,4$ or $8$, and the focal submanifolds are Veronese embeddings of $\mathbb{F}P^2$ in $S^{3m+1}(1)$, where $\mathbb{F}=\R, \C, \mathbb{H}, \mathbb{O}$ corresponding to $m=1,2,4, 8$. Meanwhile, the isoparametric hypersurfaces are tubes of constant radius over the focal submanifolds. When $g=4$, the most complicated and beautiful case, the isoparametric hypersurfaces are either of OT-FKM type or homogeneous with $(m_1, m_2) =(2, 2), (4, 5)$ (c.f. \cite{CCJ07}, \cite{Imm08}, \cite{Chi11}, \cite{Chi13}, \cite{Chi20}). When $g=6$, the multiplicities satisfy $m_1=m_2=m=1$ or $2$ and the isoparametric hypersurfaces are homogeneous (c.f. \cite{DN85}, \cite{Miy13}, \cite{Miy16}).

Now we review the isoparametric family of OT-FKM type, which holds up almost all the cases with $g=4$. Given a symmetric Clifford system $\{P_0,\cdots,P_m\}$ on $\mathbb{R}^{2l}$, \emph{i.e.} $P_{\alpha}$'s are symmetric matrices satisfying $P_{\alpha}P_{\beta}+P_{\beta}P_{\alpha}=2\delta_{\alpha\beta}I_{2l}$, Ferus, Karcher and
M\"{u}nzner (\cite{FKM81}) generalized Ozeki-Takeuchi's result (\cite{OT75}, \cite{OT76}) to construct a Cartan-M\"{u}nzner polynomial $F$ of degree $4$ on
$\mathbb{R}^{2l}$
\begin{eqnarray}\label{FKM isop. poly.}
&&\qquad F:\quad \mathbb{R}^{2l}\rightarrow \mathbb{R}\nonumber\\
&&F(x) = |x|^4 - 2\displaystyle\sum_{\alpha = 0}^{m}{\langle
P_{\alpha}x,x\rangle^2}.
\end{eqnarray} It is not difficult to verify that $f=F|_{S^{2l-1}(1)}$ is an isoparametric function on $S^{2l-1}(1)$,  which is called \emph{OT-FKM type}. The multiplicity pair is $(m_1, m_2)=(m, l-m-1)$  provided $m>0$ and $l-m-1>0$, where $l=k\delta(m)$, $k$ being a positive integer and $\delta(m)$ the dimension of the irreducible module of the Clifford algebra $\mathcal{C}_{m-1}$ which is valued as:
\vspace{1mm}

\begin{center}
\begin{tabular}{|c|c|c|c|c|c|c|c|c|c|}
\hline
$m$ & 1 & 2 & 3 & 4 & 5 & 6 & 7 & 8 & $\cdots$ $m$+8 \\
\hline
$\delta(m)$ & 1 & 2 & 4 & 4 & 8 & 8 & 8 & 8 & ~16$\delta(m)$\\
\hline
\end{tabular}
\end{center}
\vspace{1mm}

The focal submanifolds $M_+=f^{-1}(1)$ and $M_-=f^{-1}(-1)$ are minimal submanifolds with codimensions $m_1 +1$ and $m_2 +1$ in $S^{2l-1}(1)$.
It was proved in \cite{FKM81} that, when $m\not \equiv 0~ (mod~4)$, there exists exactly one kind of OT-FKM type isoparametric family.
When $m\equiv 0~(mod ~4)$, there are two kinds of OT-FKM type isoparametric families which are distinguished by $\mathrm{Trace}(P_0P_1\cdots P_m)$. Namely, the family with $P_0P_1\cdots P_m=\pm Id$, where without loss of generality we take the $+$ sign, which is called the \emph{definite} family, and the others with $P_0P_1\cdots P_m\neq\pm Id$, which are called \emph{indefinite}. There are exactly $[\frac{k}{2}]$ non-congruent indefinite families.

In this paper, we aim to give a systematic and comprehensive study on topological and geometric properties of isoparametric hypersurfaces and their focal submanifolds. The main motivation is to show that, there are still many interesting and important problems related to the theory of isoparametric hypersurfaces although the classification of isoparametric hypersurfaces in unit spheres has been accomplished.

Firstly, we focus on the topology of isoparametric family of OT-FKM type, which contain all the inhomogeneous isoparametric families.

Based on the work of \cite{FKM81}, Wang \cite{Wa88} began to study the homotopy, homeomorphism, and diffeomorphism types of isoparametric hypersurfaces and focal submanifolds of OT-FKM type. There has been given some basic results for the topology of the focal submanifold $M_-$ of OT-FKM type, for which we show further illustration in Theorem 2.1 in Section \ref{sec2}.
However, his results for the focal submanifold $M_+$ are not complete. In this paper, we firstly make progress on this problem.
According to \cite{Wa88} and \cite{QT16}, $M_+$ is an $S^{l-m-1}$-bundle over $S^{l-1}$. Denote by $\eta$ the associated vector bundle of rank $l-m$ over $S^{l-1}$. Concerning its triviality, we have a complete result.

\vspace{2mm}
\noindent\textbf{Theorem 2.2.}
The bundle \emph{$\eta$ is a trivial bundle if and only if $(m_1, m_2)=(1, 2), (2, 1),
(1, 6)$, $(6, 1)$, $(2, 5), (5, 2), (3, 4), $ or the indefinite case of $(4, 3)$.}
\vspace{2mm}

As an application of Theorem \ref{QTY2} and the work of \cite{JW54}, we show the following
\vspace{2mm}

\noindent\textbf{Theorem 2.3.}
\emph{~The focal submanifold $M_+$ of OT-FKM type is homotopy equivalent (resp. homeomorphic, diffeomorphic) to $S^{l-1}\times S^{l-m-1}$ if and only if $\eta$ is trivial, i.e., $(m_1, m_2)=(1, 2), (2, 1),
(1, 6), (6, 1), (2, 5), (5, 2), (3, 4), $ or the indefinite case of $(4, 3)$.}

As we know, the isoparametric hypersurface $M$ of OT-FKM type is diffeomorphic to $M_+\times S^m$ (c.f. \cite{FKM81}). Therefore, Theorem \ref{QTY2} and Theorem \ref{QTY3} induce the following
\vspace{2mm}

\noindent\textbf{Theorem 2.4.} 
\emph{An isoparametric hypersurface $M$ of OT-FKM type is homotopy equivalent (resp. homeomorphic, diffeomorphic) to $S^{l-1}\times S^{l-m-1}\times S^m$ if and only if $\eta$ is trivial, namely,
$(m_1, m_2)=(1, 2), (2, 1)$, $(1, 6), (6, 1), (2, 5)$, $(5, 2)$, $(3, 4), $ or the indefinite case of $(4, 3)$ .}
\vspace{2mm}

Recall the Lusternik-Schnirelmann category $cat(X)$ of a topological space $X$ that is defined to be the least number $m$ such that there exists a covering of $X$ by $m+1$ open subsets which are contractible in $X$. Notice that for a subset $A$ in $X$ which is contractible in $X$, $A$ itself may not be contractible or even not be connected with respect to the subspace topology. In \cite{DFN90}, the authors pointed out on P. 234:  \emph{Computation of $cat(X)$ presents in general a non-trivial problem, the precise value being ascertained only with great difficulty.}

As the second aspect of topological properties, we compute the Lusternik-Schnirelmann category of isoparametric hypersurfaces and focal submanifolds of OT-FKM type.
\vspace{2mm}

\noindent\textbf{Theorem 2.5.} 
\emph{Let $M, M_+, M_-$ be the isoparametric hypersurface and focal submanifolds of OT-FKM type in the unit sphere, respectively.
Then }
\emph{
\begin{itemize}
\item [(i).] $cat(M_{-})=2$; \\
$cat(M_{+})=2$, if $(m_1, m_2)\neq (8, 7), (9, 6)$, and $M_+$ is not 
in the following two cases
\begin{itemize}
\item[(a).] $(m_1, m_2)=(1, 1)$, $cat(M_+)=cat(SO(3))=3$; 
\item[(b).] $(m_1, m_2)=(4, 3)$ in the definite case, $cat(M_+)=cat(Sp(2))=3$.
\end{itemize}
\end{itemize}
\begin{itemize}
\item[(ii).] $cat(M)=3$, if $(m_1, m_2)\neq (8, 7), (9, 6)$, and $M$ is not 
in the following two cases
\begin{itemize}
\item[(a).]
$(m_1, m_2)=(1, 1)$, $cat(M)=4$;
\item[(b).] $(m_1, m_2)=(4, 3)$ in the definite case, $cat(M)=4$.
\end{itemize}
\end{itemize}}
\vspace{2mm}
\begin{rem}
It is still a problem to determine the Lusternik-Schnirelmann category of the focal submanifold $M_+$ of OT-FKM type with $(g, m_1, m_2)=(4, 8, 7)$ or $(4, 9, 6)$.
\end{rem}

Due to \cite{FKM81}, $M_-$ of OT-FKM type is an $S^{l-1}$-bundle over $S^m$. Denote the associated vector bundle by $\xi$ so $M_-$ is diffeomorphic to $S(\xi)$. Considering another topological aspect--the parallelizability of the isoparametric family of OT-FKM type,  we establish

\vspace{2mm}

\noindent\textbf{Theorem 2.6.}
\emph{~~Let $M, M_+, M_-$ be the isoparametric hypersurface and focal submanifolds of OT-FKM type, respectively.
Then
\begin{itemize}
\item[(i).] $M_-$ is parallelizable if and only if $M_-$ is s-parallelizable, if and only if
$\xi$ is trivial;
\item[(ii).] $M_+$ is parallelizable;
\item[(iii).] $M$ is parallelizable.
\end{itemize}}
\vspace{2mm}

Next, we turn to the homogeneous cases. We make a detailed investigation of topological properties of homogeneous isoparametric hypersurfaces and focal submanifolds in the unit sphere. For instance, a systematic study on the topological properties of focal submanifolds with $(g, m_1, m_2)=(4, 2, 2)$ and $(g, m_1, m_2)=(4, 4, 5)$ in Theorem \ref{QTY6} and Theorem \ref{QTY7}, respectively. As for the isoparametric families with $g=3$ and $g=6$, we deduce the following
\vspace{2mm}

\noindent\textbf{Proposition 2.1.} 
\emph{Let $M$ be the isoparametric hypersurface in $S^{3m+1}$ with $g=3$ and $m_1=m_2=m=1,2,4$ or $8$. Then
\begin{itemize}
\item[(i).] for $m=1$: $cat(M^3)=cat(SO(3)/\mathbb{Z}_2\oplus \mathbb{Z}_2)=3$;
\item[(ii).] for $m=2$: $cat(M^{6})=cat(SU(3)/T^2)=3$;
\item[(iii).] for $m=4$: $cat(M^{12})=cat(Sp(3)/Sp(1)^3)=3$;
\item[(iv).] for $m=8$: $cat(M^{24})=cat(F_4/Spin(8))=3$.
\end{itemize}}

\begin{rem}
For focal submanifolds of isoparametric hypersurfaces in unit spheres with $g=3$, it is well known that $cat(\mathbb{F}P^2)=2$ for $\mathbb{F}= \mathbb{R}, \mathbb{C}, \mathbb{H}$ or $\mathbb{O}$.
\end{rem}

\noindent\textbf{Proposition 2.2.} 
\emph{Let $M$ be the isoparametric hypersurface in $S^{6m+1}$ with $g=6$ and $m_1=m_2=m=1$ or $2$. Then
\begin{itemize}
\item[(i).] for $m=1$: $cat(M^6)=4$ and $cat(M_{\pm})=3$;
\item[(ii).] for $m=2$: $cat(M^{12})=6$ and $cat(M_{\pm})=5$.
\end{itemize}}
\vspace{2mm}

 In the second part of this paper, we investigate the intrinsic curvature properties of isoparametric hypersurfaces and focal submanifolds in unit spheres with induced metrics.  Since the curvature of the isoparametric family with $g\leq 3$ with induced metric is now very clear (which will be introduced in Section \ref{sec3}), we only need to consider the remaining cases with $g=4, 6$ and prove the following theorem as one of our main results
 \vspace{2mm}

\noindent\textbf{Theorem 3.1.} 
\emph{For the focal submanifolds of an isoparametric hypersurface in $S^{n+1}(1)$ with $g=4$ or $6$ distinct principal curvatures,
\begin{itemize}
\item[(i)] When $g=4$, the sectional curvatures of a focal submanifold with induced metric are non-negative if and only if the focal submanifold is one of the following
\begin{itemize}
\item[(a)] $M_+$ of OT-FKM type with multiplicities $(m_1, m_2)=(2,1), (6,1)$ or $(4,3)$ in the definite case;
\item[(b)] $M_-$ of OT-FKM type with multiplicities $(m_1, m_2)=(1, k)$;
\item[(c)] the focal submanifold with multiplicities $(m_1, m_2)=(2,2)$ and diffeomorphic to the oriented Grassmannian $\widetilde{G}_2(\R^5)$;
\end{itemize}
\item[(ii)] When $g=6$, the sectional curvatures of all focal submanifolds with induced metric are not non-negative.
\end{itemize}}
 \vspace{2mm}

Furthermore, we also study the Ricci curvatures of isoparametric families in Proposition \ref{Ricci of hyp}
and Proposition \ref{Ricci of focal}.

The present paper is organized as follows. In Section 2, we will study the topological properties of isoparametric hypersurfaces and focal submanifolds in the unit sphere, dividing the investigation into two subsections for the OT-FKM type and the homogeneous cases, respectively. In Section 3, we will focus on the intrinsic curvature properties of isoparametric families, concentrating on the sectional curvatures of focal submanifolds of OT-FKM type.

\section{\textbf{Topology of an isoparametric family}}\label{sec2}
In this section, we will study various topological properties of isoparametric hypersurfaces and associated focal submanifolds in unit spheres. We firstly deal with the isoparametric families of OT-FKM type which contain all the inhomogeneous cases, then deal with the homogeneous cases, including the cases with $g=3,6$ and $g=4$, $(m_1, m_2)=(2,2), (4,5)$. 

\subsection{OT-FKM type}
Given a symmetric Clifford system
$\{P_0,\cdots,P_m\}$ on $\mathbb{R}^{2l}$ with $l=k\delta(m)$, let $M, M_+, M_-$ be the associated isoparametric hypersurface of OT-FKM type in $S^{2l-1}(1)$, and the two focal submanifolds of codimension $m_1+1$ and $m_2+1$ in $S^{2l-1}(1)$, respectively.

\subsubsection{\textbf{Homotopy, homeomorphism, and diffeomorphism types of $M_-$}}\label{M_-}

As asserted by Ferus-Karcher-M\"{u}nzner \cite{FKM81} in Theorem 4.2,
$M_-$ is an $S^{l-1}$-bundle over $S^m$. Let $\xi$ be the associated vector bundle of rank $l$ over $S^m$. Then $M_-$ is diffeomorphic to $S(\xi)$, the associated sphere bundle of $\xi$. By the proof of Corollary 1 and 2 in \cite{Wa88}, it follows in fact that

\noindent\textbf{Theorem of Wang.} \emph{For the vector bundle} $\xi$, \emph{one has}

(i). \emph{If} $m=3, 5, 6$ \emph{or} $7$ ($\mathrm{mod}$ $8$), \emph{then} $\xi$ \emph{is trivial.}

(ii). \emph{If} $m=1$ or $2$ ($\mathrm{mod}$ $8$), \emph{then} $\xi$ \emph{is trivial if and only if} $k$ \emph{is even.}

(iii). \emph{If} $m=0$ ($\mathrm{mod}$ $4$), \emph{then} $\xi$ \emph{is trivial if and only if} $q=0$, \emph{where}
$$2q\delta(m)=\mathrm{Tr}(P_0P_1\cdots P_m).$$

Considering the homotopy type, homeomorphism type and diffeomorphism types of $M_-$, we have

\begin{thm}\label{QTY1} For the focal submanifold $M_-$ of OT-FKM type, we have
\begin{itemize}
\item[(i).] If $m=3, 5, 6$ or $7$ $(\mathrm{mod}$ $8)$, then $M_-$ is diffeomorphic to $S^m\times S^{l-1}$.\vspace{2mm}

\item[(ii).] If $m=1$ or $2$ $(\mathrm{mod}$ $8)$, then $M_-$ is homotopy equivalent (resp. homeomorphic, diffeomorphic) to $S^m\times S^{l-1}$ if and only if $k$ is even.\vspace{2mm}

\item[(iii).] If $m=0$ $(\mathrm{mod}$ $4)$, then $M_-$ is homotopy equivalent to $S^m\times S^{l-1}$ if and only if $q=0$ $(\mathrm{mod}$ $d_m)$, where $d_m$ is the denominator of $B_{m/4}/m$ and $B_{m/4}$ is the $(m/4)$-th Bernoulli number. Moreover, $M_-$ is homeomorphic (resp. diffeomorphic) to $S^m\times S^{l-1}$ if and only if $q=0$.
\end{itemize}
\end{thm}
\begin{proof}
Part (i) follows from part (i) of the theorem above, and part (iii) follows from Theorem 1 in \cite{Wa88}.

We only need to consider part (ii).

If $k$ is even, $\xi$ is trivial and $M_-$ is diffeomorphic to $S^m\times S^{l-1}$ by Theorem of Wang.

If $m=1$ and $k$ is odd, $\xi$ is not trivial and $M_-$ is not orientable.  Hence, $M_-$ is not homotopy equivalent to $S^1\times S^{l-1}$.

If $m>1$ and $k$ is odd, denoted by $\chi(\xi)$ the characteristic map of $\xi$. We know that $\chi(\xi)\in \pi_{m-1}O(l)$ is not zero. Since $\pi_{m-1}O(l)\cong\pi_{m-1}O$ and $\pi_{l+m-1}S^l\cong \pi^S_{m-1}$, it follows from Theorem 1.1 and 1.3 of \cite{Ad66} that $J\chi(\xi)\neq 0$, where
$$J: \pi_{m-1}O\rightarrow \pi_{m-1}^{S}$$
is the stable $J$-homomorphism. Then, by Theorem 1.11 of \cite{JW54}, $M_-$ is not homotopy equivalent to $S^m\times S^{l-1}$. Now part (ii) follows.
\end{proof}

\vspace{3mm}

\subsubsection{\textbf{Homotopy, homeomorphism, and diffeomorphism types of $M_+$}}

For a given symmetric Clifford system $\{P_0, P_1,..., P_{m}\}$ on $\mathbb{R}^{2l}$, we can choose a set of orthogonal matrices $\{E_1, E_2,..., E_{m-1}\}$ on $\mathbb{R}^l$ with the Euclidean metric, which satisfy
$E_{\alpha}E_{\beta}+E_{\beta}E_{\alpha}=-2\delta_{\alpha\beta}Id$\,\,
for $1 \leq \alpha,\beta\leq m-1$ and
\begin{equation}\label{FKM}
\begin{array}{ll}
P_0=\left(
\begin{array}{cc}
Id & 0 \\
0 & -Id \\
\end{array}\right),\;
P_1=\left(
\begin{array}{cc}
0 & Id\\
Id & 0 \\
\end{array}\right),\;\vspace{3mm}
\\
 P_{\alpha}=\left(
\begin{array}{cc}
0 & E_{\alpha-1}\\
-E_{\alpha-1} & 0 \\
\end{array}\right),
~\emph{for}~ 2 \leq \alpha\leq m.
\end{array}
\end{equation}
Then
\begin{equation}\label{M_+}
M_+=\{x\in S^{2l-1}~|~\langle P_0x, x\rangle=\langle P_1x, x\rangle=\cdots=\langle P_mx, x\rangle=0\}.
\end{equation}

To interpret the topology of $M_+$, we firstly have
\begin{lem}\label{bundle}(\cite{Wa88}, \cite{QT16})
Let $\eta$ be the subbundle of $TS^{l-1}$ such that the fiber of $\eta$ at $z\in S^{l-1}$ is the orthogonal complement in $\mathbb{R}^l$ of $m$-plane
spanned by $\{z, E_1z,..., E_{m-1}z\}$. Then $M_+$ is diffeomorphic to $S(\eta)$, the associated sphere bundle of $\eta$.
\end{lem}
\vspace{1mm}

\begin{thm}\label{QTY2}
The bundle $\eta$ is a trivial bundle if and only if $(m_1, m_2)=(1, 2), (2, 1),\\
(1, 6), (6, 1), (2, 5), (5, 2), (3, 4), $ or the indefinite case of $(4, 3)$.
\end{thm}
\vspace{1mm}

\begin{proof}
Firstly, Lemma \ref{bundle} implies that $S^{l-1}$ is parallelizable if $\eta$ is trivial. According to \cite{Ad62}, if $\eta$ is trivial, then $l=2, 4$ or $8.$
Since $m_1=m>0$ and $m_2=l-m-1>0$, it follows that $l\geq m+2\geq 3$. Clearly, when $m\geq 7$, $l\geq m+2\geq 9$, and thus $\eta$ is not trivial.

To complete the proof, we still need to consider the case $m\leq 6$.

If $m=1$, then $l=k\delta(m)=k$ and $\eta$ is isomorphic to $TS^{l-1}$. For this case, $\eta$ is trivial if and only if $l=k=4$ or $8$. Namely,  $(m_1, m_2)=(1, 2)$ or $(1, 6)$.

If $m=2$, then $l=k\delta(2)=2k$, $k\geq 2$ and $M_+$ is diffeomorphic to the complex Stiefel manifold $U(k)/U(k-2)$ (c.f. \cite{TXY12}). According to \cite{JW54}, if $U(k)/U(k-2)$ is of the same homotopy type with $S^{2k-1}\times S^{2k-3}$, then $\pi_{4k-1}(S^{2k})$ contains an element whose Hopf invariant is unity. Furthermore, by \cite{Ad60}, it implies that $k=2$ or $4$. Hence, if $\eta$ is trivial, then $k=2$ or $4$. Namely,  $(m_1, m_2)=(2, 1)$ or $(2, 5)$.
Conversely, for the focal submanifold $M_+$ of OT-FKM type with $(m_1, m_2)=(2, 1)$ or $(2, 5)$,
it is clear that $\eta$ is trivial.

If $m=3$, then $l=k\delta(3)=4k$ and $k\geq 2$. Thus, if $\eta$ is trivial, then $k=2$. Namely,  $(m_1, m_2)=(3, 4)$. Conversely, for the focal submanifold $M_+$ of OT-FKM type with $(m_1, m_2)=(3, 4)$, it follows from $\pi_{6}SO(5)=0$ that $\eta$ is trivial.

If $m=4$, then $l=k\delta(4)=4k$ and $k\geq 2$. For this case, $\eta$ is trivial implies $k=2$, that is, $(m_1, m_2)=(4, 3)$. Now, we need to consider the number $q$ which is defined by $q\delta(m)=\mathrm{Tr}(P_0P_1\cdots P_m)$.
When $k=2$ and $q=2$, i.e. the definite case of $(4, 3)$, $M_+$ is diffeomorphic to $Sp(2)$ and is not the same homotopy type as $S^7\times S^3$. It means that $\eta$ is not trivial for the definite case.
When $k=2$ and $q=0$, i.e. the indefinite case of $(4, 3)$, the symmetric Clifford system can be extended, which implies that $\eta$ is trivial (c.f. \cite{QTY13}, \cite{QT16})

If $m=5$, then $l=k\delta(5)=8k$ and $k\geq 1$. For this case, $\eta$ is trivial implies $k=1$, that is, $(m_1, m_2)=(5, 2)$. Conversely, for the focal submanifold $M_+$ of OT-FKM type with $(m_1, m_2)=(5, 2)$, 
the symmetric Clifford system can be extended, which implies that $\eta$ is trivial.

If $m=6$, then $l=k\delta(6)=8k$ and $k\geq 1$. Similar to the case $m=5$, we have $\eta$ is trivial if and only if $k=1$, i.e. $(m_1, m_2)=(6, 1)$.
\end{proof}

\vspace{1mm}

\begin{thm}\label{QTY3}
The focal submanifold $M_+$ of OT-FKM type is homotopy equivalent (resp. homeomorphic, diffeomorphic) to $S^{l-1}\times S^{l-m-1}$ if and only if $\eta$ is trivial.
\end{thm}

\begin{proof}
If $\eta$ is trivial, then it is clear that $M_+$ is homotopy equivalent (resp. homeomorphic, diffeomorphic) to $S^{l-1}\times S^{l-m-1}$.
Consequently, the ``only if" part is left to be considered. To complete the work, we will prove that if $M_+$ is homotopy equivalent to $S^{l-1}\times S^{l-m-1}$, then $\eta$ is trivial.

Given a symmetric Clifford system
$\{P_0,\cdots,P_m\}$ on $\mathbb{R}^{2l}$ with $l=k\delta(m)$, as written in (\ref{M_+}), we have $$M_+=\{x\in S^{2l-1}(1)~
|~\langle P_0x, x\rangle=\langle P_1x, x\rangle
=\cdots=\langle P_mx, x\rangle=0\}.$$
Define
$$M_1:=\{x\in S^{2l-1}(1)~
|~\langle P_0x, x\rangle=\langle P_1x, x\rangle
=0\}.$$
Let $\eta_1$ be the tangent bundle of $S^{l-1}$. Then $M_1=S(\eta_1)$. By virtue of Lemma \ref{bundle}, we obtain
$M_+\cong S(\eta)$ and $\eta_1\cong \eta \oplus \bm{\varepsilon}^{m-1}$, where $\bm{\varepsilon}^{m-1}$ is the trivial bundle of rank $m-1$.

According to \cite{JW54}, $M_1$ has the same homotopy type with $S^{l-1}\times S^{l-2}$ if and only if $\pi_{2l-1}S^l$ contains an
element whose Hopf invariant is unity. On the other hand, it follows from \cite{Ad60} that there is no element of Hopf invariant one in $\pi_{2l-1}S^l$ unless
$l=1,2,4$ or $8$. Equivalently, the Whitehead product $[\iota_{l-1}, \iota_{l-1}]$ is non-zero in the homotopy group $\pi_{2l-3}S^{l-1}$
unless $l=1,2,4$ or $8$. In our case, $l-m-1>0$ implies that $l\geq3$. Hence, if $l\neq4$ and $l\neq8$, then
$$J\chi(\eta_1)=-[\iota_{l-1}, \iota_{l-1}]\neq 0$$
where $J:\pi_{l-2}SO(l-1)\rightarrow \pi_{2l-3}S^{l-1}$ is the $J$-homomorphism, and $\chi(\eta_1)$ is the characteristic map of $\eta_1$.
So we obtain the following commutative diagram up to a sign (c.f. (1.2) and (1.3) of \cite{JW54})
\begin{equation*}
\begin{tikzcd}
\pi_{l-2}SO(l-m) \ar{r}{i_{*}} \ar{d}{J}& \pi_{l-2}SO(l-1) \ar{d}{J} \\
\pi_{2l-m-2}S^{l-m} \ar{r}{\Sigma^{m-1}}& \pi_{2l-3}S^{l-1}
\end{tikzcd}
\end{equation*}
where $i_{*}$ is the induced homomorphism of the canonical inclusion $i: SO(l-m)\rightarrow SO(l-1)$, and $\Sigma^{m-1}$ is the iterated suspension. Then it follows that
$$(\Sigma^{m-1}\circ J)(\chi(\eta))=\pm (J\circ i_{*})(\chi(\eta))=\pm J\chi(\eta_1)\neq 0.$$
Therefore, $M_+$ is not homotopy equivalent to $S^{l-1}\times S^{l-m-1}$ by Theorem 1.11 in the part I of \cite{JW54}.

Now we are in a position to consider the two special cases $l=4$ and $8$:

If $l=4$, then $4=k\delta(m)$ and $(m, k)=(1, 4)$ or $(2, 2)$, i.e. $(m_1, m_2)=(1, 2)$ or $(2,1)$.
For these two cases, by Theorem \ref{QTY2}, $\eta$ is trivial and $M_+$ is always homotopy equivalent to $S^{l-1}\times S^{l-m-1}$.

If $l=8$, then $8=k\delta(m)$ and $(m, k)=(1, 8), (2, 4), (3, 2), (4, 2), (5, 1), (6, 1)$, i.e., $(m_1, m_2)=(1, 6), (2, 5), (3, 4), (4, 3), (5, 2), (6, 1)$.  Except for the $(4, 3)$ case, it follows from Theorem \ref{QTY2} that $\eta$ is trivial and $M_+$ is homotopy equivalent to $S^{l-1}\times S^{l-m-1}$. Hence, we need only to check the case $(m_1, m_2)=(4, 3)$.
If $(m_1, m_2)=(4, 3)$ and $q=0$, then $M_+$ has the same homotopy type with $S^3 \times S^7$ and $\eta$ is trivial. If $(m_1, m_2)=(4, 3)$ and $q=2$, then $M_+$ is not homotopy equivalent to $S^3 \times S^7$ and $\eta$ is not trivial.

Now the proof of Theorem \ref{QTY3} is complete.
\end{proof}


\vspace{3mm}


\subsubsection{\textbf{Homotopy, homeomorphism, and diffeomorphism types of $M$}}\hspace{2mm}\label{2.1.3}
Let $M$ be an isoparametric hypersurface of OT-FKM type with $(m_1, m_2)=(m, l-m-1)$.
As asserted in Theorem 4.2 of \cite{FKM81}, the normal bundle of $M_+$ is trivial and the isoparametric hypersurface $M$ is a trivial sphere bundle over $M_+$, that is, $M$ is diffeomorphic to $M_+\times S^m$. Therefore, the topology of $M_+$ has a more direct bearing from $M$ than that of $M_-$.
Conversely, with the foreshadowing of Theorem \ref{QTY2} and \ref{QTY3}, we get the following conclusion

\vspace{2mm}

\begin{thm}\label{QTY4}
An isoparametric hypersurface $M$ of OT-FKM type is homotopy equivalent (resp. homeomorphic, diffeomorphic) to $S^{l-1}\times S^{l-m-1}\times S^m$ if and only if $\eta$ is trivial, namely,
$(m_1, m_2)=(1, 2), (2, 1)$, $(1, 6), (6, 1), (2, 5), (5, 2), (3, 4), $ or the indefinite case of $(4, 3)$.
\end{thm}

\begin{proof}
If $\eta$ is trivial, it follows clearly that $M$ is homotopy equivalent (resp. homeomorphic, diffeomorphic) to $S^{l-1}\times S^{l-m-1}\times S^m$. To finish the proof, we will show if $M$ is homotopy equivalent to $S^{l-1}\times S^{l-m-1}\times S^m$, then $\eta$ is trivial. The following fact will be used

If $X, Y$ are $1$-connected CW-complexes such that
$$H^*(X; \mathbb{Z})\cong H^*(Y; \mathbb{Z})\cong H^*(S^{l-1}\times S^{l-m-1}; \mathbb{Z})$$
with respect to the cohomology ring structures, and $m\neq l-m-1$, then
$X$ is homotopy equivalent to $Y$ (denoted by $X\simeq Y$) if and only if $X\times S^m\simeq Y \times S^m$.

For readers' convenience, a proof is given as follows. We only need to show if $X\times S^m\simeq Y \times S^m$, then $X\simeq Y$.
Let $f: X\times S^m\rightarrow Y \times S^m$ be a given homotopy equivalence. Choosing a fixed point $p_0$ in $S^m$, we get a natural inclusion
$i: X \rightarrow X\times S^m, x\mapsto (x, p_0)$. Moreover, let $\pi_{Y}: Y \times S^m \rightarrow Y, (y, p)\mapsto y$ be the projection. Now, considering the composition of maps $X \overset{i}{\rightarrow} X\times S^m \overset{f}{\rightarrow} Y \times S^m \overset{\pi_{Y}}{\rightarrow} Y$, we get a map $\varphi: X\rightarrow Y$. Let $\widetilde{\alpha}, \widetilde{\beta}$ be the generators of $H^{m}(X\times S^m; \mathbb{Z}), H^{l-m-1}(X\times S^m; \mathbb{Z})$ respectively, and $\widetilde{\gamma}$ be one of the generators of $H^{l-1}(X\times S^m; \mathbb{Z})\cong \mathbb{Z}\oplus \mathbb{Z}$ corresponding to $S^{l-1}$. Similarly, let $\alpha, \beta$ be the generators of $H^{m}(Y\times S^m; \mathbb{Z}), H^{l-m-1}(Y\times S^m; \mathbb{Z})$ respectively, and $\gamma$ be one of the generators of $H^{l-1}(Y\times S^m; \mathbb{Z})\cong \mathbb{Z}\oplus \mathbb{Z}$ corresponding to $S^{l-1}$. Then from K\"{u}nneth formula, $\widetilde{\alpha}\widetilde{\beta}$ and $\widetilde{\gamma}$
are exactly two generators of $H^{l-1}(X\times S^m; \mathbb{Z})$. Moreover, $\alpha\beta$ and $\gamma$ are exactly two generators of $H^{l-1}(Y\times S^m; \mathbb{Z})$. Let $f^{*}: H^*(Y\times S^m; \mathbb{Z})\rightarrow H^*(X\times S^m; \mathbb{Z})$ be the induced homomorphism. Since $f: X\times S^m\rightarrow Y \times S^m$ is a homotopy equivalence, it follows that $f^{*}$ is an isomorphism. It implies that $f^{*}(\alpha)=\pm \widetilde{\alpha}$ and $f^{*}(\beta)=\pm \widetilde{\beta}$. Assume $f^{*}(\gamma)=k_1\widetilde{\alpha}\widetilde{\beta}+k_2\widetilde{\gamma}$ for some integers $k_1$ and $k_2$. Since $f^{*}: H^{l-1}(Y\times S^m; \mathbb{Z})\rightarrow H^{l-1}(X\times S^m; \mathbb{Z})$ is an isomorphism and $f^{*}(\alpha\beta)=\pm \widetilde{\alpha}\widetilde{\beta}$,
it follows that $k_2=\pm 1$. Meanwhile, the induced homomorphisms of $i$ and $\pi_{Y}$ are clear. Consequently, we obtain that the induced homomorphism $\varphi^{*}: H^k(X; \mathbb{Z})\rightarrow H^k(Y; \mathbb{Z})$ is isomorphic for any integer $k$. Due to $\pi_1 X\cong \pi_1 Y\cong0$ and Whitehead's theorem, then $\varphi$ is a homotopy equivalence between $X$ and $Y$. We finish the proof of the fact above.

If $m=l-m-1>0$, then $m=l-m-1=1$. Then $M=SO(3)\times S^1$ and is not homotopy equivalent to $S^2\times S^1\times S^1$. Moreover, $\eta$ is not trivial for this case. Hence, the theorem is true for this special case. If $m\neq l-m-1$, by the fact above, $M\simeq S^{l-1}\times S^{l-m-1}\times S^m$ if and only if $M_+\simeq S^{l-1}\times S^{l-m-1}$. Thus by applying Theorem \ref{QTY2} and \ref{QTY3} for the homotopy type of $M_+$, we conclude that $M\simeq S^{l-1}\times S^{l-m-1}\times S^m$ if and only if $\eta$ is trivial, namely,
$(m_1, m_2)=(1, 2), (2, 1)$, $(1, 6), (6, 1), (2, 5), (5, 2), (3, 4), $ or the indefinite case of $(4, 3)$ .

\end{proof}


\subsubsection{\textbf{Lusternik-Schnirelmann category}}
Let $X$ be a topological space, by definition, the Lusternik-Schnirelmann category $cat(X)$ of $X$ is the least number $m$ such that there exists a covering of $X$ by $m+1$ open subsets which are contractible in $X$. In this part, we will determine the Lusternik-Schnirelmann category for isoparametric hypersurfaces and focal submanifolds of OT-FKM type.

\vspace{2mm}

\begin{thm}\label{QTY5}
Let $M, M_+, M_-$ be the isoparametric hypersurface and focal submanifolds of OT-FKM type in the unit sphere, respectively.
Then
\begin{itemize}
\item [(i).] $cat(M_{-})=2$; \\
$cat(M_{+})=2$, if $(m_1, m_2)\neq (8, 7), (9, 6)$, and $M_+$ is not 
in the following two cases
\begin{itemize}
\item[(a).] $(m_1, m_2)=(1, 1)$, $cat(M_+)=cat(SO(3))=3$; 
\item[(b).] $(m_1, m_2)=(4, 3)$ in the definite case, $cat(M_+)=cat(Sp(2))=3$.
\end{itemize}
\vspace{2mm}

\item[(ii).] $cat(M)=3$, if $(m_1, m_2)\neq (8, 7), (9, 6)$, and $M$ is not 
in the following two cases
\begin{itemize}
\item[(a).]
$(m_1, m_2)=(1, 1)$, $cat(M)=cat(SO(3)\times S^1)=4$;
\item[(b).] $(m_1, m_2)=(4, 3)$ in the definite case, $cat(M)=cat(Sp(2)\times S^4)=4$.
\end{itemize}
\end{itemize}
\end{thm}

\begin{proof}
We start with the proof for part (i).

As we mentioned before, the focal submanifold $M_-$ of OT-FKM type is an $S^{l-1}$-bundle over $S^m$. Since $m_1=m>0$ and $m_2=l-m-1>0$, comparing with the assumption in Fact 2.2 of \cite{Iw03}, one has $r:=l-1\geq 2$ and $t:=m-1<r$. Thus $cat(M_-)=2$.

As for $M_+$ of OT-FKM type, we will firstly prove the following

\vspace{2mm}
\noindent
\textbf{Fact}: If $1<m<l-m-1$, then $cat(M_+)=2$.
\vspace{2mm}

Due to \cite{Mun81}, we have
$$H^{k}(M_+; \mathbb{Z})\cong\left\{\begin{array}{ll} \mathbb{Z},\,\,\, \emph{if} ~k=0, l-m-1, l-1~ \emph{or}~ 2l-m-2, \\  0, \,\,\, \emph{otherwise.}
\end{array}\right.$$
By Lemma \ref{bundle} and the long exact sequence of homotopy groups, $\pi_1(M_+)=0$. Then the universal coefficient theorem and Hurewicz theorem induce 
$$\pi_{i}(M_+)\cong0,~\,\,\,\emph{for}\,~i<l-m-1,\quad\emph{and}\quad
\pi_{l-m-1}(M_+)\cong \mathbb{Z}.$$
Therefore, by virtue of Proposition 5.1 in \cite{Ja78}, we obtain
$$cat(M_+)\leq \frac{2l-m-2}{l-m-1}<3.$$
On the other hand, as asserted in \cite{DFN90},
$$cat(M_+)\geq cuplength(M_+)=2,$$
where cuplength$(M_+)$ is the cohomological length of $M_+$. Hence $cat(M_+)=2$ if $1<m<l-m-1$.

So we are left to consider the cases for $m=1$ and $m\geq l-m-1$.
When $m=1$, recalling the expression of $M_+$ in (\ref{M_+}), it is easily seen that
$M_+$ is diffeomorphic to the Stiefel manifold $V_2(\mathbb{R}^l)\cong SO(l)/SO(l-2)$. Thus by \cite{Ni07} and \cite{DFN90}, we obtain that
$$cat(M_+)=\left\{\begin{array}{ll}cat(V_2(\mathbb{R}^l))=2,\,\,~~  \emph{if}~ l\geq4,\vspace{2mm}\\
cat(SO(3))=3, \,\,~~  \emph{if}~ l=3.
\end{array}
\right.$$

When $m\geq l-m-1$, by the assumption that $(m_1, m_2)\neq(8, 7)$ or $(9, 6)$, there only exist the following cases:
$$(m_1, m_2)=(1, 1), (2, 1), (4, 3), (5, 2), (6, 1).$$
We have dealt with the case $(m_1, m_2)=(1, 1)$ in the case $m=1$. As pointed out by \cite{FKM81}, the isoparametric families of OT-FKM type with multiplicities $(2, 1), (6, 1), (5, 2)$ and the indefinite case of $(4, 3)$ are congruent to those with multiplicities $(1, 2), (1, 6), (2, 5)$ and $(3, 4)$, and the focal submanifolds $M_{\pm}$ coincide with the corresponding $M_{\mp}$. Thus by the previous result on $cat(M_-)$, we get $cat(M_+)=2$ if $(m_1, m_2)=(2, 1), (6, 1), (5, 2)$ and the indefinite case of $(4, 3)$. 
For the definite case of $(4, 3)$, we know that $M_+$ is diffeomorphic to $Sp(2)$ (c.f. \cite{QTY13}, \cite{TY15}), thus $cat(M_+)=cat(Sp(2))=3$ (c.f. \cite{Sc65}). The proof of part (i) is complete.

Next, we continue to prove part (ii).

As we mentioned in the beginning of subsection \ref{2.1.3}, $M\cong M_+\times S^m$.
Then $cat(M)=cat(M_+\times S^m)$, which is equal to $cat(M_+)$ or $cat(M_+)+1$ (c.f. \cite{Iw03}). Furthermore, according to \cite{Mun81}, 
$$H^{*}(M; R)\cong H^*(S^{l-1}\times S^{l-m-1}\times S^m; R),$$ 
where $R=\mathbb{Z}$ if both of focal submanifolds are orientable, and $R=\mathbb{Z}_2$ otherwise. Then it follows from \cite{DFN90} that $cat(M)\geq cuplength(M)=3$.

If $M$ is neither in the case $(m_1, m_2)=(1,1)$, nor in the definite case of $(4, 3)$, then $cat(M_+)=2$, as we proved in part (i), which implies that $cat(M)\leq cat(M_+)+1=3$. Hence $cat(M)=3$.

If $(m_1, m_2)=(1, 1)$, then $M\cong S^1\times SO(3)$. Considering $SO(3)$ as the total space of $S^1$-bundle over $S^2$, by Fact 2.1 in \cite{Iw03}, $t=r=1$ and $\alpha=\pm 2$, we have $cat(M)=cat(S^1\times SO(3))=4$.

For the definite case of $(4, 3)$, $M\cong S^4\times Sp(2)$. Thus $cat(M)=4$ by Theorem 3.2 of \cite{St99}.
\end{proof}
\vspace{4mm}

\subsubsection{\textbf{Parallelizability of the isoparametric family of OT-FKM type}}
Recall that a smooth manifold $X$ is said to be \emph{parallelizable} if its tangent bundle $TX$ is trivial, and $X$ is said to be \emph{s-parallelizable} if the
Whitney sum of its tangent bundle $TX$ and a trivial line bundle is trivial. Obviously, a parallelizable manifold is s-parallelizable. For the parallelizability of the isoparametric family of OT-FKM type, we have the following
\vspace{2mm}

\begin{thm}\label{s-parallel}
Let $M, M_+, M_-$ be the isoparametric hypersurface and focal submanifolds of OT-FKM type. Then
\begin{itemize}
\item[(i).] $M_-$ is parallelizable if and only if $M_-$ is s-parallelizable, if and only if
$\xi$ is trivial;

\item[(ii).] $M_+$ is parallelizable.

\item[(iii).] $M$ is parallelizable.
\end{itemize}
\end{thm}

\begin{proof}
Firstly, we will prove part (i) for $M_-$.

As we mentioned in subsection \ref{M_-}, $M_-$ is an $S^{l-1}$-bundle over $S^m$ with the projection $\pi: M_-\rightarrow S^m$. Let $\xi$ be the associated vector bundle of rank $l$ over $S^m$. Then $M_-$ is diffeomorphic to $S(\xi)$.

If $m=1$ and $l=k$ is even, then $M_-$ is diffeomorphic to $S^{1}\times S^{l-1}$ by Theorem \ref{QTY1}, and thus parallelizable. If $m=1$ and $l=k$ is odd,
then $M_-$ is not orientable, not parallelizable and even not s-parallelizable.

From now on, we focus on the cases for $m>1$. By the Theorem 1.1 in \cite{Su64}, $M_-$ is s-parallelizable if and only if $\pi^{*}\xi$ represents a zero element in $\widetilde{KO}(M_-)$. Since $l>m+1$, there exists a nowhere zero section of $\xi$, equivalently speaking, there exists a map $s: S^m\rightarrow M_-$ such that $\pi \circ s=Id_{S^m}$. Then it follows that the induced homomorphism
$$\pi^{*}: \widetilde{KO}(S^m)\rightarrow\widetilde{KO}(M_-)$$
is injective. That is to say, $M_-$ is s-parallelizable if and only if $\pi^{*}\xi$ is stably trivial, if and only if $\xi$ is stably trivial.

On the other hand, since $l\geq m+2$, $\xi$ is stably trivial if and only if $\xi$ is trivial. Hence, $M_-$ is s-parallelizable if and only if $\xi$ is trivial. Furthermore, by Theorem 1.3 in \cite{Su64} and the fact that $l=k\delta(m)$ is even when $m>1$, we get $M_-$ is s-parallelizable if and only if $M_-$ is parallelizable. The proof of part (i) is now complete.

Next, we continue to prove part (ii) for $M_+$.

By Lemma \ref{bundle}, $M_+\cong S(\eta)$. Clearly, $\eta$ is stably trivial as a subbundle of $TS^{l-1}$. Then $M_+$ is s-parallelizable by Theorem 1.1 in \cite{Su64}, and thus parallelizable by Theorem 1.3 in \cite{Su64}.

Consequently,  $M\cong M_+\times S^m$ is also parallelizable. We now complete the proof of Theorem \ref{s-parallel}.
\end{proof}


For an isoparametric family of OT-FKM type, although the normal bundle of $M_+$ in $S^{2l-1}$ is always trivial, it is still a natural problem to determine when the normal
bundle of $M_-$ in $S^{2l-1}$ is trivial. As an application of Theorem \ref{s-parallel}, we obtain

\begin{thm}\label{normal bundle of M-}
Given an isoparametric family of OT-FKM type, the normal bundle of $M_-$ in $S^{2l-1}$ is trivial if and only if
$(m_1, m_2)=(1, 2), (2, 1), (1, 6), (6, 1),$\\
$(2, 5), (5, 2), (3, 4),$
or the indefinite case of $(4, 3).$
\end{thm}

\begin{proof}
Let $\nu M_-$ be the normal bundle of $M_-$ in $S^{2l-1}$ so that
$$TM_-\oplus\nu M_-=TS^{2l-1}|_{M_-}.$$
Since $TS^{2l-1}$ is stably trivial,
$\nu M_-$ is stably trivial if and only if $TM_-$ is stably trivial. Moreover, by Theorem \ref{s-parallel}, $TM_-$ is stably trivial if and only if
$\xi$ is trivial. Thus, $\nu M_-$ is stably trivial if and only if $\xi$ is trivial.

Now we assume that $\nu M_-$ is trivial, which implies the triviality of $\xi$. Then it follows that
$$M\cong S^{l-m-1}\times M_-\cong S^{l-m-1}\times S^m\times S^{l-1},$$
which implies further that $\eta$ is trivial by Theorem \ref{QTY3}. Therefore, if $\nu M_-$ is trivial, then  $(m_1, m_2)=(1, 2), (2, 1),
(1, 6), (6, 1), (2, 5), (5, 2), (3, 4),$ or the indefinite case of $(4, 3)$ by Theorem \ref{QTY2}.

Conversely, according to \cite{FKM81}, the families with multiplicities $(2, 1), (6, 1)$, $(5, 2)$ are congruent to those with multiplicities $(1, 2), (1, 6), (2, 5)$. Moreover, the indefinite family with multiplicities $(4, 3)$ is congruent to the family with multiplicities $(3, 4)$. Thus the focal submanifolds $M_{\pm}$ are congruent to corresponding $M_{\mp}$. Recalling that the normal bundle of $M_+$ in $S^{2l-1}$ is always trivial, we finish the proof.
\end{proof}

\vspace{3mm}

\subsection{Homogeneous case}
In this subsection, we study the topology of homogeneous isoparametric hypersurfaces and focal submanifolds. It is well known that a homogeneous (isoparametric) hypersurface in the unit sphere can be characterized as a principal orbit of the isotropy representation of some rank two symmetric space $G/K$, while focal submanifolds correspond to the singular orbits
(c.f. \cite{HL71}, \cite{TT72}).

\subsubsection{\textbf{The case with $(g, m_1, m_2)=(4, 2, 2)$.}}
Consider the lie algebra $so(5, \mathbb{R})$ and the adjoint representation of $SO(5)$ on it.
Then the principal orbits of this action constitute the homogeneous 1-parameter family
of isoparametric hypersurfaces in $S^9$ with $(g, m_1,
m_2)=(4, 2, 2)$. Each isoparametric hypersurface $M$ is diffeomorphic to $SO(5)/T^2$, and the two
focal submanifolds $M_{\pm}$ are diffeomorphic to $\mathbb{C}P^3$ and the oriented Grassmann manifold $\widetilde{G}_2(\mathbb{R}^5)$, respectively (c.f. \cite{QTY13}).

\begin{thm} \label{QTY6}

$\mathrm{(i)}.$ For the focal submanifold $M_+$ with $(g, m_1, m_2)=(4, 2, 2)$, which is diffeomorphic to $\mathbb{C}P^3$,

\begin{itemize}
\item[(a).] $M_+$ is an $S^2$-bundle over $S^4$;

\item[(b).] The cohomology ring $H^{*}(M_+; \mathbb{Z})$ is not isomorphic to $H^*(S^2\times S^4; \mathbb{Z})$;

\item[(c).] $cat(M_+)=3$;

\item[(d).] $M_+$ is not s-parallelizable.\vspace{1mm}
\end{itemize}

$\mathrm{(ii)}.$ For the focal submanifold $M_-$ with $(g, m_1, m_2)=(4, 2, 2)$, which is diffeomorphic to the oriented Grassmann manifold $\widetilde{G}_2(\mathbb{R}^5)$,

 \begin{itemize}
\item[(a).] $M_-$ is not an $S^p$-bundle over $S^q$ for any positive integers $p$ and $q$;
\item[(b).] The cohomology ring $H^*(M_-; \mathbb{Z})$ is not isomorphic to $H^*(S^2\times S^4; \mathbb{Z})$;

\item[(c).] $cat(M_-)=3$;
\item[(d).] $M_-$ is not s-parallelizable.

\end{itemize}
\end{thm}

\begin{proof}
Firstly, note that $M_+\cong\mathbb{C}P^3$ is the total space of the twistor bundle of almost complex structures on $S^4$ with fiber $\mathbb{C}P^1$. It is clear that (a), (b) and (d) of part (i) are valid. Moreover, since $M_+$ admits a K\"{a}hler metric and $\pi_1(M_+)=0$, one has $cat(M_+)=\mathrm{dim}_{\mathbb{C}}M_+=3$ (c.f. \cite{Be76}), which verifies (c) of part (i).

Next, we continue to prove part (ii). According to \cite{Mun81},
$$H^{k}(M_-; \mathbb{Z})\cong \left\{\begin{array}{ll}\mathbb{Z},\,\, \emph{if}\,\, k=0, 2, 4, 6\\
0, \,\, \emph{ otherwise}.
\end{array}\right.$$
Consequently, if $M_-$ is the total space of a sphere bundle over a sphere, then the only possible cases are $S^4$-bundle over $S^2$ or $S^2$-bundle over $S^4$.

Suppose $M_-$ is an $S^4$-bundle over $S^2$.
Considering the fiber bundle $S^1\hookrightarrow V_2(\mathbb{R}^5)\rightarrow \widetilde{G}_2(\mathbb{R}^5)$, by virtue of the homotopy exact sequence of a fiber bundle, one obtains that $\pi_3(M_-)\cong\pi_3(V_2(\mathbb{R}^5)\cong\mathbb{Z}_2$ (c.f. \cite{St51}). On the other hand, by assumption, there exists an exact sequence
$$\cdots\rightarrow\pi_3S^4 \rightarrow \pi_3 M_-\rightarrow \pi_3 S^2 \rightarrow\pi_2 S^4\rightarrow\cdots$$
which leads to $\pi_3 M_-\cong\pi_3 S^2\cong \mathbb{Z}$. Clearly, this contradicts $\pi_3(M_-)\cong\mathbb{Z}_2$.

Now we consider the other remaining possible case. Suppose that $M_-$ is an $S^2$-bundle over $S^4$ with projection $\pi$, and let $\xi$ be the associated vector bundle of rank $3$ over $S^4$. Then
$$TM_-\oplus \bm{\varepsilon}\cong \pi^{*}TS^4\oplus \pi^*\xi.$$
It follows that the total Stiefel-Whitney class $W(TM_-)=\pi^*W(\xi)$. But $W(\xi)=1+w_4(\xi)$, which implies that $W(TM_-)=1+w_4(TM_-)$ and thus $w_2(TM_-)=0$. However, by Lemma 2.4 of \cite{MM82}, $w_2(TM_-)\neq 0$, which leads to a contradiction. The proof for (a) of part (ii) is now complete.

Besides, from the proof of Proposition 3.4 in \cite{Ta95}, we see that the cohomology ring $H^*(M_-; \mathbb{Z})$ is different from $H^*(S^2\times S^4; \mathbb{Z})$, which verifies (b) of part (ii). Moreover,
observing that $\widetilde{G}_2(\mathbb{R}^5)$ is a Hermitian symmetric space and $\pi_1(\widetilde{G}_2(\mathbb{R}^5))=0$, we obtain $cat(M_-)=3$ following \cite{Be76}, which verifies (c) of part (ii).
At last, (d) of part (ii) follows from Corollary 1.2 of \cite{MM82}.

Now we complete the proof for Theorem \ref{QTY6}.
\end{proof}

\begin{rem}
\begin{itemize}
\item[(i)] The cohomology rings $H^{*}(M_+^6; \mathbb{Z})$ and $H^*(M_-^6; \mathbb{Z})$ are not isomorphic. \vspace{2mm}
\item[(ii)]
By Theorem \ref{QTY6}, $M_+^6$ and $M_-^6$ are not s-parallelizable. Thus, the normal bundles of $M_{\pm}^6$ in $S^9$ are not stably trivial.
\end{itemize}
\end{rem}

\vspace{3mm}

\subsubsection{\textbf{The case with $(g, m_1, m_2)=(4, 4, 5)$.}}

Consider the Lie algebra $so(5, \mathbb{C})$. The unitary group $U(5)$
acts on it by the adjoint representation $g\cdot Z=\overline{g}Zg^{-1}$
for $g\in U(5)$ and $Z\in so(5, \mathbb{C})$. The principal orbits of this action constitute the
homogeneous 1-parameter family of isoparametric hypersurfaces in $S^{19}$ with multiplicities $(m_1, m_2)=(4,5)$.
The two focal submanifolds $M_{\pm}$ are diffeomorphic to $U(5)/(Sp(2)\times U(1))$ and $U(5)/(SU(2)\times U(3))$, respectively.

\vspace{2mm}

\begin{thm}\label{QTY7}
$\mathrm{(i)}.$ For the focal submanifold $M_+$ with $(g, m_1, m_2)=(4, 4, 5)$, which is diffeomorphic to $U(5)/(Sp(2)\times U(1))$,

 \begin{itemize}
\item[(a).] $M_+$ is an $S^q$-bundle over $S^p$ if and only if $p=9$ and $q=5$;

\item[(b).] the normal bundle of $M_+$ in $S^{19}$ is not trivial;

\item[(c).] $cat(M_+)=2$.

\end{itemize}
\vspace{1mm}

$\mathrm{(ii)}.$ For the focal submanifold $M_-$ with $(g, m_1, m_2)=(4, 4, 5)$, which is diffeomorphic to $U(5)/(SU(2)\times U(3))$,

\begin{itemize}
\item[(a).] $M_-$ is not an $S^q$-bundle over $S^p$ for any positive integers $p$ and $q$;

\item[(b).] the normal bundle of $M_-$ in $S^{19}$ is not stably trivial;

\item[(c).] $2\leq cat(M_-)\leq 3$.

\end{itemize}

\end{thm}

\begin{proof}
We start with the proof for part (i). Since
$$M_+\cong U(5)/Sp(2)\times U(1)\cong SU(5)/S(Sp(2)\times U(1))\cong SU(5)/Sp(2),$$
we see $M_+$ fibers over $SU(5)/SU(4)\cong S^9$ with fiber $SU(4)/Sp(2)\cong S^5$. It follows immediately that $\pi_1(M_+)=0$. We will show that there is no other possibility for $M_+$ to be a sphere bundle over a sphere.

According to \cite{Mun81},
$$H^{k}(M_+; \mathbb{Z})\cong\left\{\begin{array}{ll} \mathbb{Z},~~\emph{if}~~k=0, 5, 9, 14\\
0,~~\emph{ otherwise}.
\end{array}\right.$$
Thus Hurewicz theorem leads to $\pi_5(M_+)\cong H^{5}(M_+; \mathbb{Z})\cong\mathbb{Z}$.

Suppose $M_+$ is the total space of an $S^{14-p}$-bundle over $S^{p}$. Then there exists an exact sequence
$$\cdots\rightarrow\pi_6 S^p\rightarrow \pi_5S^{14-p}\rightarrow \pi_5 M_+\rightarrow\pi_5 S^p\rightarrow \pi_4 S^{14-p}\rightarrow\cdots.$$
If $p\geq 10$, then $\pi_5S^{14-p}\cong\pi_5 M_+\cong \mathbb{Z}$, which contradicts the facts $\pi_5S^4\cong\pi_5S^3\cong\pi_5S^2\cong\mathbb{Z}_2$ (c.f. \cite{St51}) and $\pi_5S^1=0$. If $p<9$, then $14-p>5$ and $\pi_5S^p\cong\pi_5M_+\cong\mathbb{Z}$, which implies that $p=5$. Namely, $M_+$ might be the total space of an $S^{9}$-bundle over $S^{5}$.
Suppose $M_+$ is the total space of a certain $S^9$-bundle over $S^5$. Since $\pi_4SO(10)\cong 0$ (c.f. \cite{St51}), the characteristic map is trivial in $\pi_4SO(10)$, which implies that $M_+$ is diffeomorphic to $S^5\times S^9$. However, it is impossible by Proposition 1.1 in \cite{TY15}. Therefore, $M_+$ could only be the total space of an $S^{5}$-bundle over $S^{9}$.

As for (b) of part (i), we argue by contradiction. Suppose that the normal bundle of $M_+$ in $S^{19}$ is trivial. Then the isoparametric hypersurface $M^{18}$
is diffeomorphic to $M_+^{14}\times S^4$. From the fibration $S^1\hookrightarrow M_-\rightarrow G_2(\mathbb{C}^5)$, we derive
$$\pi_{10}M_-\cong\pi_{10}G_2(\mathbb{C}^5).$$
Considering the fibration $U(2)\hookrightarrow V_2(\mathbb{C}^5)\rightarrow G_2(\mathbb{C}^5)$, we have the following exact sequence
$$\cdots\rightarrow\pi_{10}S^3\rightarrow\pi_{10}V_2(\mathbb{C}^5)\rightarrow\pi_{10}G_2(\mathbb{C}^5)\rightarrow \pi_9S^3\rightarrow\cdots.$$
By Lemma II.1 (ii) in \cite{Ke60},
\begin{equation}\label{pi10V}
\pi_{10}V_2(\mathbb{C}^5)\cong \mathbb{Z}_{12}.
\end{equation}
Moreover, from P. 332 of \cite{Hu59}, $$\pi_9S^3\cong\mathbb{Z}_3.$$
Combining these together, we obtain an upper bound of
$|\pi_{10}G_2(\mathbb{C}^5)|$, the order of $\pi_{10}G_2(\mathbb{C}^5)$:
$$|\pi_{10}G_2(\mathbb{C}^5)|\leq 36.$$
Moreover, the isoparametric hypersurface $M^{18}$ is also an $S^5$-bundle over $M_-$, which induces an exact sequence
\begin{equation}\label{fgh}
\cdots\rightarrow\pi_{10}S^5\overset{f}{\rightarrow} \pi_{10}M \overset{g}{\rightarrow}\pi_{10} M_-\overset{h}{\rightarrow} \pi_{9}S^5\rightarrow\cdots.
\end{equation}
It is proved in \cite{Hu59} that $\pi_9S^5\cong \pi_{10}S^5\cong \mathbb{Z}_2$ and $\pi_{10}S^4\cong \mathbb{Z}_{24}\oplus \mathbb{Z}_2$.
Thus $$\pi_{10}M=\pi_{10}M_+\oplus \pi_{10}S^4\cong \pi_{10}M_+\oplus \mathbb{Z}_{24}\oplus \mathbb{Z}_2,$$
and furthermore, we derive from the exact sequence (\ref{fgh}) that
$$|\pi_{10}M|=|\pi_{10}M_+\oplus\mathbb{Z}_{24}\oplus \mathbb{Z}_2|\leq |\pi_{10}S^5||\pi_{10}M_-|\leq 72.$$
It follows that $\pi_{10}M_+\cong0$ and $f$ is not a zero homomorphism. By the exact sequence (\ref{fgh}) again, $\mathrm{Ker} g= \mathrm{Im} f\cong \mathbb{Z}_2$, $\mathrm{Im} g\cong \pi_{10}M/\mathrm{Ker} g$, $\mathrm{Ker} h=\mathrm{Im} g$ and $\mathrm{Im} h\cong \pi_{10}M_-/\mathrm{Ker} h$. That is,
\begin{equation*}\label{fgh'}
0{\rightarrow}\mathbb{Z}_2\overset{f}{\rightarrow} \mathbb{Z}_{24}\oplus \mathbb{Z}_2 \overset{g}{\rightarrow}\pi_{10} M_-\overset{h}{\rightarrow} \mathbb{Z}_2.
\end{equation*}
Then it follows from $\pi_{10}M=\mathbb{Z}_{24}\oplus \mathbb{Z}_2$ and $|\pi_{10}M_-|=|\pi_{10}G_2(\mathbb{C}^5)|\leq 36$ that $h$ is a zero homomorphism. Hence, $\pi_{10}M_-\cong \mathbb{Z}_{24}\oplus \mathbb{Z}_2/\mathrm{Ker} g$ and $|\pi_{10}M_-|=24$.

Considering the fibration $U(2)\hookrightarrow V_2(\mathbb{C}^5)\rightarrow G_2(\mathbb{C}^5)$ again, we have the exact sequence
$$\cdots\rightarrow\pi_{10}S^3\overset{\alpha}{\rightarrow}\pi_{10}V_2(\mathbb{C}^5)\overset{\beta}{\rightarrow}\pi_{10}G_2(\mathbb{C}^5)\overset{\gamma}{\rightarrow} \pi_9S^3\rightarrow\cdots$$
with $\pi_{10}S^3\cong \mathrm{Z}_{15}$(c.f. \cite{Hu59}), $\pi_{10}V_2(\mathbb{C}^5)\cong \mathbb{Z}_{12}$, $|\pi_{10}M_-|=|\pi_{10}G_2(\mathbb{C}^5)|=24$ and $\pi_{9}S^3\cong\mathbb{Z}_3$. It is clear that $\gamma$ is not a zero homomorphism. Since $|\pi_{9}S^3|=3$ is a prime number, we infer that $\gamma$ is surjective. Hence, $$\pi_{10}G_2(\mathbb{C}^5)/\mathrm{Ker} \gamma\cong \mathbb{Z}_3,$$ and $$|\mathrm{Im} \beta|=|\mathrm{Ker} \gamma|=|\pi_{10}G_2(\mathbb{C}^5)|/3=8.$$
Furthermore, from $\pi_{10}V_2(\mathbb{C}^5)/\mathrm{Ker} \beta \cong \mathrm{Im} \beta$, we derive that
$$|\pi_{10}V_2(\mathbb{C}^5)|/|\mathrm{Ker} \beta|=|\mathrm{Im} \beta|=8,$$ which contradicts (\ref{pi10V}).

As for (c) of part (i), from $M_+$ is an $S^5$-bundle over $S^9$, it follows that $\pi_iM_+=0$ for $i<5$. Using Proposition 5.1 in \cite{Ja78}, we obtain $cat(M_+)\leq \frac{14}{5}$. On the other hand, it is well-known that a closed manifold $X$ satisfying $cat(X) = 1$ is homotopy equivalent to a sphere. Hence, $cat(M_+)\geq 2$, and thus $cat(M_+)=2$.

Next, we are going to prove part (ii). As for (a) in part (ii), similar to the arguments of part (i), it is only possible that $M_-$ is an $S^9$-bundle over $S^4$ or an $S^4$-bundle over
$S^9$.

Suppose $M_-$ is an $S^9$-bundle over $S^4$. On one hand, from the proof for (b) of part (i), we see $\pi_{10}M_-\cong\pi_{10}G_2(\mathbb{C}^5)$ and
$|\pi_{10}G_2(\mathbb{C}^5)|\leq 36$.
On the other hand, the assumption that $M_-$ is an $S^9$-bundle over $S^4$ implies the following exact sequence $$\cdots\rightarrow\pi_{10}M_-\rightarrow\pi_{10}S^4\rightarrow\pi_9S^9\rightarrow\cdots.$$
Since $\pi_{10}S^4\cong \mathbb{Z}_{24}\oplus\mathbb{Z}_2$ (c.f. \cite{Hu59}) and $\pi_9S^9\cong\mathbb{Z}$, we deduce that the homomorphism from $\pi_{10}S^4$ to $\pi_9S^9$ is zero. That is to say, the homomorphism from $\pi_{10}M_-$ to $\pi_{10}S^4$ is onto. It follows that $$|\pi_{10}G_2(\mathbb{C}^5)|=|\pi_{10}M_-|\geq |\pi_{10}S^4|=48,$$
which contradicts $|\pi_{10}G_2(\mathbb{C}^5)|\leq 36$. Hence, $M_-$ is not an $S^9$-bundle over $S^4$.

Suppose $M_-$ is an $S^4$-bundle over $S^9$. Let $\xi$ be the associated vector bundle over $S^9$ of rank $5$. Then
$$TM_-\oplus \bm{\varepsilon}\cong \pi^{*}TS^9\oplus \pi^*\xi.$$
It follows that the total Stiefel-Whitney class $W(TM_-)=\pi^*W(\xi)=1$ where $W(\xi)=1$. However, Lemma 1.1 in \cite{Ta91} reveals that $w_4(TM_-)\neq0$, which leads to a contradiction. Hence, $M_-$ is not an $S^4$-bundle over $S^9$.

Clearly, (b) of part (ii) follows from $w_4(TM_-)\neq0$. Now, we are in a position to consider (c) of part (ii). Since $M_-$ is not homotopy equivalent to a sphere, one obtains $cat(M_-)\geq 2$. Moreover, $M_-$ admits a Morse function with $4$ critical points(c.f. \cite{Ta91}), then we have $cat(M_-)\leq 3$(c.f. \cite{DFN90}).
\end{proof}

\vspace{2mm}

\begin{rem}
$\mathrm{(i)}.$ By Theorem B in \cite{Fa99}, $M_+^{14}$ is almost diffeomorphic to an $S^5$-bundle over $S^9$. We reveal by Theorem \ref{QTY7} that $M_+^{14}$ is indeed diffeomorphic to an $S^5$-bundle over $S^9$. Moreover, according to \cite{TY15}, $M_+^{14}$ is not homotopy equivalent to $S^5\times S^9$, which implies that the $S^5$-bundle over $S^9$ is not trivial.

$\mathrm{(ii)}.$ It is still a problem to determine the exact value of $cat(M_-^{13})$ in the theorem above.
Recall that if $X$ is a compact topological manifold of dimension $n$,
the ball-category of X, $ballcat(X)$, is the minimal number $m$ such that there is a
covering of $X$ with $m+1$ closed disks of dimension $n$. By Theorem 3.46 in \cite{CLOT03}, $cat(M_-^{13})=ballcat(M_-^{13})$.
If $cat(M_-^{13})=3$, then $wcat(M_-)=cat(M_-)$ by Theorem 2.2 in \cite{St99}. For the definition of $wcat$, see \cite{Ja78}. In general, $wcat(X)\leq cat(X)$ for any finite CW complex.
\end{rem}

\vspace{3mm}

\subsubsection{\textbf{The $g=3$ case}}
\begin{prop}\label{QTY8}
Let $M$ be the isoparametric hypersurface in $S^{3m+1}$ with $g=3$ and $m_1=m_2=m=1,2,4$ or $8$. Then
\begin{itemize}
\item[(i).] $m=1$: $cat(M^3)=cat(SO(3)/\mathbb{Z}_2\oplus \mathbb{Z}_2)=3$;\vspace{1mm}

\item[(ii).] $m=2$: $cat(M^{6})=cat(SU(3)/T^2)=3$;\vspace{1mm}

\item[(iii).] $m=4$: $cat(M^{12})=cat(Sp(3)/Sp(1)^3)=3$;\vspace{1mm}

\item[(iv).] $m=8$: $cat(M^{24})=cat(F_4/Spin(8))=3$.
\end{itemize}
\end{prop}

\begin{proof}
When $g=3$ and $m=1$, $M^3=SO(3)/\mathbb{Z}_2\oplus \mathbb{Z}_2$.
Since $\pi_1(M^3)=\mathbb{Q}_8=\{\pm \mathrm{1}, \pm \mathrm{i}, \pm \mathrm{j}, \pm \mathrm{k}\}$ (c.f. \cite{GH87}), we see
$\pi_1(M^3)\neq 0$ and not free. Thus $cat(M^3)=3$ (c.f. \cite{GG92}).

When $g=3$ and $m=2$, we have $cat(M^{6})=cat(SU(3)/T^2)=3$ (c.f. \cite{Si75}).

When $g=3$ and $m=4$, due to \cite{Mun81} and the Hurewicz theorem, we have $\pi_iM^{12}=0$ for $i\leq3$. It follows from Proposition 5.1 in \cite{Ja78} that $cat(M^{12})\leq 3$. Suppose $cat(M^{12})\leq 2$, then it follows from Theorem 7.6 in \cite{Ta68} that
$$c(sq^4): H^4(M^{12}; \mathbb{Z}_2)\rightarrow H^8(M^{12}; \mathbb{Z}_2)$$
is zero, where $c(sq^4)=sq^4+sq^3sq^1$. By \cite{Mun81},
$$H^4(M^{12}; \mathbb{Z}_2)\cong H^4(M_+; \mathbb{Z}_2)\oplus H^4(M_-; \mathbb{Z}_2)\cong\mathbb{Z}_2\oplus\mathbb{Z}_2,$$
and
$$H^8(M^{12}; \mathbb{Z}_2)\cong H^8(M_+; \mathbb{Z}_2)\oplus H^8(M_-; \mathbb{Z}_2)\cong\mathbb{Z}_2\oplus\mathbb{Z}_2.$$
Moreover, we have the following isomorphisms of the cohomology rings with $\mathbb{Z}_2$-coefficients
$$H^*(M_+; \mathbb{Z}_2)\cong H^*(M_-; \mathbb{Z}_2)\cong H^*(\mathbb{H}P^2; \mathbb{Z}_2).$$
Thus
\begin{eqnarray*}
c(sq^4)=sq^4: ~~H^4(M^{12}; \mathbb{Z}_2)&\rightarrow& H^8(M^{12}; \mathbb{Z}_2)\\ x&\mapsto& sq^4(x)=x\cup x
 \end{eqnarray*}
is not zero, which leads to a contradiction. Therefore, $cat(M^{12})=3$.

When $g=3$ and $m=8$, replacing $c(sq^4)$ with $c(sq^8)$, a similar argument as in the case $g=3, m=4$ by Theorem 7.6 of \cite{Ta68} implies $cat(M^{24})=3$.

The proof for Proposition \ref{QTY8} is now complete.
\end{proof}

\vspace{2mm}

\subsubsection{\textbf{The $g=6$ case}}
\begin{prop}\label{QTY9}
Let $M$ be the isoparametric hypersurface in $S^{6m+1}$ with $g=6$ and $m_1=m_2=m=1$ or $2$. Then
\begin{itemize}
\item[(i).] $m=1$: $cat(M^6)=4$ and $cat(M_{\pm})=3$;\vspace{1mm}

\item[(ii).] $m=2$: $cat(M^{12})=6$ and $cat(M_{\pm})=5$.
\end{itemize}
\end{prop}

\begin{proof}
First, we consider the case $g=6$ and $m_1=m_2=1$. By the observation of Miyaoka , an isoparametric hypersurface with $g=6, m=1$ is the inverse image of an isoparametric hypersurface with $g=3, m=1$ via the Hopf fibration (c.f. \cite{Miy93}). Thus
$$M^6\cong S^3\times SO(3)/\mathbb{Z}_2\oplus \mathbb{Z}_2, \quad M_+^5\cong M_-^5\cong S^3\times \mathbb{R}P^2.$$
From $cat(\mathbb{R}P^2)=2$, one has $cat(M_{\pm})=2$ or $3$. On the other hand, according to \cite{DKR08}, $\pi_1(M_{\pm}^5)=\mathbb{Z}_2$ is not free, then
$cat(M_{\pm})\geq 3$. Therefore, $cat(M_{\pm}^5)=3$. Moreover, using the fact $cat(SO(3)/\mathbb{Z}_2\oplus \mathbb{Z}_2)=3$, it follows that $cat(M^6)=4$
by Theorem 3.8 of \cite{Ru99}.

Next, we consider the case $g=6$ and $m_1=m_2=2$. For this case,
$$M^{12}\cong G_2/T^2, \quad M_+^{10}\cong M_-^{10}\cong G_2/U(2).$$
From Theorem 2 of \cite{Si75}, it follows that $cat(M^{12})=6$. On the other hand, observe that $M_{\pm}\cong G_2/U(2)$ is diffeomorphic to $\widetilde{G}_2(\mathbb{R}^7)$ which is a simply connected Hermitian symmetric space (c.f. \cite{Miy11}). Hence we obtain $cat(M_{\pm})=5$ by \cite{Be76}.

The proof for Proposition \ref{QTY9} is now complete.
\end{proof}
\vspace{4mm}

\section{\textbf{Curvatures of an isoparametric family}}\label{sec3}
For an isoparametric family in the unit sphere $S^{n+1}(1)$, it is known that the scalar curvatures of isoparametric hypersurfaces are constant and non-negative (c.f. \cite{Tan04}, \cite{TY20}). Besides, for every unit normal vector at any point of a focal submanifold, the corresponding shape operator has principal curvatures $\cot\frac{(j-i)\pi}{g}$ (for a certain $1\leq i\leq g$) with multiplicity $m_j$, for $j\neq i, 1\leq j\leq g$ (c.f. Corollary 3.22 in \cite{CR15}). Thus a direct calculation by virtue of Gauss equation leads to the fact that the scalar curvatures of focal submanifolds are constant and non-negative.
Therefore, we only consider the sectional curvatures and Ricci curvatures in this section.

\subsection{Sectional curvature}

As we mentioned before, the principal curvatures of an isoparametric hypersurface $M^n\subset S^{n+1}(1)$ with $g$ distinct principal curvatures can be written as $\lambda_k=\cot(\theta+\frac{k-1}{g}\pi)$ ($k=1,\ldots, g$) with $\theta\in (0,\frac{\pi}{g})$.

It is obvious that an isoparametric hypersurface $M^n$ with $g=1$ has positive sectional curvatures and that with $g=2$ has non-negative sectional curvatures. When $g\geq 3$, taking $e_1, e_g \in T_pM^n$ to be unit principal vectors corresponding to distinct principal curvatures $\lambda_1, \lambda_g$, it follows easily from Gauss equation that the sectional curvature 
$$K(e_1, e_g)=1+\lambda_1\lambda_g=-\frac{\cot\frac{\pi}{g}(1+\lambda_1^2)}{\lambda_1-\cot\frac{\pi}{g}}<0.$$ 
Thus the sectional curvature of $M^n$ is not non-negative.

As for sectional curvatures of the focal submanifolds, we know that when $g=1$, the focal submanifolds are just two points. When $g=2$, the focal submanifolds are isometric to
$S^p(1)$ or $S^{n-p}(1)$, thus the sectional curvatures are constant and positive. When $g=3$, as we mentioned in the introduction, $m_1=m_2=m=1, 2, 4$ or $8$, and the focal submanifolds are Veronese embedding of $\mathbb{F}P^2$ in $S^{3m+1}$, where $\mathbb{F}=\R, \C, \mathbb{H}, \mathbb{O}$ corresponding to $m=1, 2, 4, 8$. The induced metric of $\R P^2$ has constant sectional (Gaussian) curvature $K=\frac{1}{3}$, and the induced metric of $\C P^2$, $\mathbb{H}P^2$ or $\mathbb{O}P^2$ is symmetric with sectional curvature $\frac{1}{3}\leq K\leq \frac{4}{3}$ (c.f. \cite{TY13}). Thus the sectional curvatures of these focal submanifolds are positive (c.f. \cite{Zil14}).

For the focal submanifolds with $g=4$ and $6$, we can determine which of them has non-negative sectional curvatures with respect to the induced metric from $S^{n+1}(1)$ as follows

\begin{thm}\label{1}
For the focal submanifolds of an isoparametric hypersurface in $S^{n+1}(1)$ with $g=4$ or $6$ distinct principal curvatures, we have
\begin{itemize}
\item[(i)] When $g=4$, the sectional curvatures of a focal submanifold with induced metric are non-negative if and only if the focal submanifold is one of the following
\begin{itemize}
\item[(a)] $M_+$ of OT-FKM type with multiplicities $(m_1, m_2)=(2,1), (6,1)$ or $(4,3)$ in the definite case;
\item[(b)] $M_-$ of OT-FKM type with multiplicities $(m_1, m_2)=(1, k)$;
\item[(c)] the focal submanifold with multiplicities $(m_1, m_2)=(2,2)$ and diffeomorphic to $\widetilde{G}_2(\R^5)$;
\end{itemize}

\item[(ii)] When $g=6$, the sectional curvatures of focal submanifolds with induced metric are not non-negative.
\end{itemize}
\end{thm}

\begin{rem}
The list of focal submanifolds with non-negative sectional curvatures in Theorem \ref{1} is the same as that of Ricci parallel focal submanifolds in \cite{TY15} and \cite{LY15}.
\end{rem}

\begin{proof}
According to the classification of isoparametric hypersurfaces in unit spheres, we will divide our proof of Theorem \ref{1} into several parts from 3.1.1 to 3.1.6.

\noindent

\subsubsection{ \textbf{$M_+$ of OT-FKM type.}}

Recall that the focal submanifold $M_+$ of OT-FKM type with $(m_1, m_2)=(m, l-m-1)$ can be described as
\begin{equation}\label{M+}
  M_+=\{x\in S^{2l-1}(1)~|~\langle P_0x, x\rangle=\cdots=\langle P_mx, x\rangle=0\}.
\end{equation}
As pointed out by \cite{FKM81},  $\{P_0x,P_1x, \ldots, P_mx\}$ is an orthonormal basis of the normal space $T_x^{\bot}M_+$ at $x\in M_+\subset S^{2l-1}(1)$ and for any normal vector $\xi_{\alpha}=P_{\alpha}x$ $(\alpha=0,...,m)$,  the corresponding shape operator $A_{\alpha}:=A_{\xi_{\alpha}}$ for any $X\in T_xM_+$ is $A_{\alpha}X=-(P_{\alpha}X)^T$, the tangential component of $-P_{\alpha}X$.

It follows from Gauss equation that
\begin{eqnarray}\label{Gauss}
K(X, Y)&=& 1+\sum_{\alpha=0}^m\langle A_{\alpha}X,
X\rangle\langle A_{\alpha}Y, Y\rangle -\sum_{\alpha=0}^m \langle A_{\alpha}X,
Y\rangle^2\nonumber
\\
&=&1+\sum_{\alpha=0}^m\langle P_{\alpha}X, X\rangle\langle P_{\alpha}Y, Y\rangle - \sum_{\alpha=0}^m\langle P_{\alpha}X, Y\rangle^2.
\end{eqnarray}
We first show a lemma as below to eliminate most cases of $M_+$

\begin{lem}\label{lemma 1}
If $l>2m$, for each point $p\in M_+$, there exists a tangent plane with negative sectional curvature.
\end{lem}

\begin{proof}
As in \cite{FKM81}, we can define $(P_0,...,P_m)$ in the symmetric Clifford system by
$$P_0(u,v):=(u,-v),~P_1(u,v):=(v,u),~P_{1+{\alpha}}(u,v):=(E_{\alpha}v,-E_{\alpha}u),~u,v\in \mathbb{R}^l.$$
where $E_1,...,E_{m-1}$ are skew-symmetric endomorphisms of  $\R^l$ with
$E_{\alpha}E_{\beta}+E_{\beta}E_{\alpha}=-2\delta_{\alpha\beta}Id.$

Thus for any point $z=(z_1, z_2)\in S^{2l-1}(1)\subset\R ^l\oplus \R^l$, $z\in M_+$ if and only if
$$ |z_1|^2=|z_2|^2=\frac{1}{2}, ~~ \langle z_1, z_2\rangle=0,~~~~~ \langle E_{\alpha} z_1, z_2\rangle=0,\,\,\,\, \forall ~\alpha=1,\ldots,m-1. $$
Moreover,  $X=(x_1, x_2)\in T_zM_+$ if and only if
\begin{equation}\label{tangent vector of M+}
\left\{\begin{array}{ll}
\langle x_1, z_1\rangle=\langle x_2, z_2\rangle=0,\\
\langle x_1, z_2\rangle+\langle x_2, z_1\rangle=0,\\
\langle x_1, E_{\alpha}z_2\rangle-\langle x_2, E_{\alpha}z_1\rangle=0, \quad\forall~\alpha=1,\ldots,m-1.
\end{array}
\right.
\end{equation}

Now we take $X=(c,0)$ and $Y=(0, c)\subset\R ^l\oplus \R^l$. It follows easily from (\ref{tangent vector of M+})
that $X, Y\in T_zM_+$ if and only if
\begin{equation}\label{XY}
\langle c, ~z_1\rangle=\langle c, ~z_2\rangle=0, ~~{\rm{and}}~~ \langle c, ~E_{\alpha}z_1\rangle=\langle c, ~E_{\alpha}z_2\rangle=0,  ~~\forall~\alpha=1,\ldots,m-1.
\end{equation}
Obviously, when $l>2m$, we can always choose a unit vector $c\in\R^l$ such that (\ref{XY}) is fulfilled, and furthermore,
\begin{equation*}
\left\{\begin{array}{ll}
P_0X=X,~~ P_0Y=-Y, ~~\langle P_1X, ~Y\rangle=1,\\
\langle P_{\alpha}X, ~X\rangle=\langle P_{\alpha}Y, ~Y\rangle=0, \quad\forall ~~\alpha=1,\ldots, m,\\
\langle P_{\alpha}X, ~Y\rangle=0, \quad \forall~~ \alpha\neq 1.
\end{array}
\right.
\end{equation*}
Therefore, the sectional curvature
$$K(X, Y)=1+\langle P_0X, X\rangle\langle P_0Y, Y\rangle - \langle P_1X, Y\rangle^2=-1<0.$$
\end{proof}

With Lemma \ref{lemma 1} in mind, analyzing the conditions $m\geq 1$, $k\delta(m)-m-1\geq 1$ and $l=k\delta(m)\leq 2m$, we find that there are only the following cases to consider:
$$(m_1,m_2)=(2,1), (4,3), (5,2), (6,1), (8,7), \rm{and } ~~(9,6).$$

According to \cite{FKM81},  the families with $(m_1, m_2)=(2, 1), (6, 1), (5, 2)$ and the indefinite one of $(4, 3)$-families are congruent to those with $(m_1, m_2)=(1, 2)$, $(1, 6)$, $(2, 5)$ and $(3, 4)$, respectively. In these cases, we can leave the proof to 3.1.2.
But a direct proof reveals more geometric properties sometimes.
\vspace{2mm}

\noindent
$(1)$  \emph{\textbf{$M_+$ with $(m_1, m_2)=(2, 1)$}.} In this case, it can be seen directly that $M_+$ is isometric to $U(2)$ with a bi-invariant metric, thus the sectional curvature is non-negative.
\vspace{2mm}


\noindent
$(2)$  \emph{\textbf{$M_+$ with $(m_1, m_2)=(4, 3)$ in the definite case}. }
According to \cite{QT16}, $M_+$ in this case is isometric to $Sp(2)$ with a bi-invariant metric, thus the sectional curvature is non-negative.

\vspace{2mm}

\noindent
$(3)$  \emph{\textbf{$M_+$ with $(m_1, m_2)=(4, 3)$ in the indefinite case}. }
Setting $P = P_0P_1P_2P_3$, it is easy to see that $P$ is symmetric and $P^2$ = $Id.$ Then following Theorems 5.1 and 5.2 in \cite{FKM81}, we can find a point $x\in M_+$ as the +1-eigenvector of $P$ , i.e. $P_0P_1P_2P_3x = x.$ Then the tangent space $T_xM_+$ can be decomposed into
$T_xM_+=\mathcal{V}_x\oplus \mathcal{W}_x$ with (c.f. \cite{TY15})
$$\mathcal{V}_x=Span\{P_0P_1x, P_0P_2x, P_0P_3x, P_0P_4x, P_1P_4x, P_2P_4x, P_3P_4x\}$$ and
$$\mathcal{W}_x=Span\{ P_0P_1P_4x, P_0P_2P_4x, P_0P_3P_4x\}.$$
Let
$$X:=(P_0P_1P_4x+P_1P_4x)\big/\sqrt{2},\quad Y:=(P_0P_2P_4x-P_2P_4x)\big/\sqrt{2}.$$
Then we obtain
\begin{equation*}
\left\{\begin{array}{ll}
|X|=|Y|=1, \quad \langle X, ~Y\rangle=0\\
P_0X=X,~~ P_0Y=-Y, ~~\langle P_3X, ~Y\rangle=-1,\\
\langle P_{\alpha}X, ~X\rangle=\langle P_{\alpha}Y, ~Y\rangle=0, \quad\forall ~~\alpha\neq 0.\\
\langle P_{\alpha}X, ~Y\rangle=0\quad \forall~~ \alpha\neq 3.
\end{array}
\right.
\end{equation*}
Therefore, the sectional curvature
$$K(X, Y)=1+\langle P_0X, X\rangle\langle P_0Y, Y\rangle - \langle P_3X, Y\rangle^2=-1<0.$$

\noindent
$(4)$  \emph{\textbf{$M_+$ with $(m_1, m_2)=(5, 2)$}. }
Choose $x\in S^{15}(1)$ as a common eigenvector of the commuting $4$-products
$P_0P_1P_2P_3$ and $P_0P_1P_4P_5$. Without loss of generality, we assume $P_0P_1P_4P_5x=x$.  According to \cite{TY15}, the tangent space $T_xM_+$ can be decomposed into
$T_xM_+=\mathcal{V}_x\oplus \mathcal{W}_x$ with
$$\mathcal{V}_x=Span\{P_0P_1x, P_0P_2x, P_0P_3x, P_0P_4x, P_0P_5x, P_2P_4x, P_2P_5x\}$$ and
$$\mathcal{W}_x=Span\{ P_0P_2P_4x, P_0P_2P_5x\}.$$
Let
$$X:=(P_0P_2P_4x-P_2P_4x)\big/\sqrt{2},\quad Y:=(P_0P_2P_5x+P_2P_5x)\big/\sqrt{2}.$$
Then we have
\begin{equation*}
\left\{\begin{array}{ll}
|X|=|Y|=1, \quad \langle X, ~Y\rangle=0\\
P_0X=-X,~~ P_0Y=Y, ~~\langle P_1X, ~Y\rangle=1,\\
\langle P_{\alpha}X, ~X\rangle=\langle P_{\alpha}Y, ~Y\rangle=0, \quad\forall ~~\alpha\neq 0,\\
\langle P_{\alpha}X, ~Y\rangle=0\quad \forall~~ \alpha\neq 1.
\end{array}
\right.
\end{equation*}
and thus the sectional curvature
$$K(X, Y)=1+\langle P_0X, X\rangle\langle P_0Y, Y\rangle - \langle P_1X, Y\rangle^2=-1<0.$$

\vspace{2mm}

\noindent
\emph{\textbf{An alternative proof for $M_+$ in the indefinite $(4,3)$ case.}}
For any symmetric Clifford system $P_0,\ldots, P_5$ on $\R^{16}$ in the $(5,2)$ case, we can remove any one of $P_0,\ldots, P_5$ to obtain a new  symmetric Clifford system on $\R^{16}$, which is indefinite. The reason is that any definite symmetric Clifford system with $Q_0\cdots Q_4=\pm Id$ can not be extended to $Q_5$.

As in the case (4), we still choose $x\in S^{15}(1)$ as a common eigenvector of the commuting $4$-products $P_0P_1P_2P_3$ and $P_0P_1P_4P_5$. Then we remove any one of $P_2, P_3, P_4, P_5$ from $P_0,\ldots, P_5$, and still obtain that $x\in M_+$ of the indefinite $(4,3)$ family. Choosing the same $X, Y$ as in Case (3), we still obtain $K(X, Y)=-1<0$.

We'll take a similar approach to the $(8.7)$ indefinite case later.

\vspace{2mm}

\noindent
$(5)$  \emph{\textbf{$M_+$ with $(m_1, m_2)=(6, 1)$}.} In this case, we need only to consider $M_-$ of OT-FKM type with $(m_1, m_2)=(1, 6)$. As showed by \cite{TY13}, $M_-$ with $(m_1, m_2)=(1, 6)$ is isometric to $(S^1(1)\times S^7(1))\big/\mathbb{Z}_2$, thus the sectional curvature is non-negative.
\vspace{2mm}

\noindent
$(6)$  \emph{\textbf{$M_+$ with $(m_1, m_2)=(8, 7)$ in the definite case}. }
For $(u,v)\in \R^{32}=\mathbb{O}^4$ and $u=(u_1, u_2), v=(v_1, v_2)\in \mathbb{O}\oplus\mathbb{O}$, we construct a symmetric Clifford system $P_0,\ldots, P_8$ on $\R^{32}$ as follows
\begin{equation*}
P_0(u ,v)=(u, -v),\quad, P_1(u, v)=(v, u),\quad P_{1+\alpha}(u, v)=(E_{\alpha}v, -E_{\alpha}u),
\end{equation*}
where $E_{\alpha}$ acts on $u$ or $v$ in this way:
$$E_{\alpha}u=(e_{\alpha}u_1, e_{\alpha}u_{2}),~\alpha=1,\cdots,7,$$
and $\{1,e_1,e_2,\cdots, e_7\}$ is the standard orthonormal basis of the Octonions (Cayley numbers) $\mathbb{O}$.

In fact, let $1=(1, 0), e_1=(i, 0), e_2=(j, 0), e_3=(k, 0), e_4=(0, 1), e_5=(0, i), e_6=(0, j), e_7=(0, k)\in\mathbb{H}\times\mathbb{H}$.
Recalling the Cayley-Dickson construction of the product of Octonions $\mathbb{O}\cong \mathbb{H}\times\mathbb{H}$:
\begin{eqnarray*}
\mathbb{O}\times \mathbb{O}&\longrightarrow&\mathbb{O}\\
(a,b),~(c,d)&\mapsto&(a,b)\cdot (c,d) =: (ac-\bar{d}b, ~ da+b\bar{c}),
\end{eqnarray*}
 one can see easily that $e_1(e_2(\cdots(e_7z)))=-z$, $\forall~ z\in\mathbb{O}$. Then it
follows immediately that $P_0, \ldots, P_8$ is a definite system.

Take $x=\frac{1}{\sqrt{2}}(1, 0, 0, 1)$, and
$$X:=\frac{1}{2}(e_2, e_3, e_3, -e_2), \quad Y:=\frac{1}{2}(-e_7, e_6, e_6, e_7).$$
Clearly, $x\in M_+$ and $X, Y\in T_xM_+$. Moreover
\begin{equation*}
\left\{\begin{array}{ll}
|X|=|Y|=1, \quad \langle X, ~Y\rangle=0\\
P_2X=-X,~~ P_2Y=Y, ~~\langle P_5X, ~Y\rangle=1,\\
\langle P_{\alpha}X, ~X\rangle=\langle P_{\alpha}Y, ~Y\rangle=0, \quad\forall ~~\alpha\neq 2,\\
\langle P_{\alpha}X, ~Y\rangle=0\quad \forall~~ \alpha\neq 5.
\end{array}
\right.
\end{equation*}
Therefore, the sectional curvature
$$K(X, Y)=1+\langle P_2X, X\rangle\langle P_2Y, Y\rangle - \langle P_5X, Y\rangle^2=-1<0.$$

\vspace{2mm}

To facilitate the expression, we deal with the case $(9, 6)$ before the indefinite $(8,7)$ case.
\vspace{2mm}

\noindent
$(7)$  \emph{\textbf{$M_+$ with $(m_1, m_2)=(9, 6)$}. }

For $(u,v)\in \R^{32}=\mathbb{O}^4$ and $u=(u_1, u_2), v=(v_1, v_2)\in \mathbb{O}\oplus\mathbb{O}$, we construct a symmetric Clifford system $P_0,\ldots, P_9$ on $\R^{32}$ as follows
\begin{equation*}
\left\{\begin{array}{ll}
P_0(u ,v)=(u, -v), ~~P_1(u, v)=(v, u),\\
P_{1+\alpha}(u, v)=(E_{\alpha}v, -E_{\alpha}u) ~(\alpha=1,\ldots, 7),\\
P_9(u, v)=(Jv, -Ju).
\end{array}
\right.
\end{equation*}
$E_{\alpha}$ acts on $u$ or $v$ in this way
$$E_{\alpha}u=(e_{\alpha}u_1, -e_{\alpha}u_{2}),~\alpha=1,\cdots,7,$$
where $\{1,e_1,e_2,\cdots, e_7\}$ is the standard orthonormal basis of the Octonions $\mathbb{O}$
and $J$ acts on $u, v$ by
 $$Ju=J(u_1, u_2):=(u_2, -u_1).$$

Take $x=\frac{1}{\sqrt{2}}(1, 0, 0, e_1)$, and
$$X:=(e_2, 0, 0, 0), \quad Y:=(0, 0, 0, e_2).$$
It is easy to see that $x\in M_+$ and $X, Y\in T_xM_+$. Moreover
\begin{equation*}
\left\{\begin{array}{ll}
|X|=|Y|=1, \quad \langle X, ~Y\rangle=0\\
P_0X=X,~~ P_0Y=-Y, ~~\langle P_9X, ~Y\rangle=1,\\
\langle P_{\alpha}X, ~X\rangle=\langle P_{\alpha}Y, ~Y\rangle=0, \quad\forall ~~\alpha\neq 0,\\
\langle P_{\alpha}X, ~Y\rangle=0\quad \forall~~ \alpha\neq 9.
\end{array}
\right.
\end{equation*}
Therefore, the sectional curvature
$$K(X, Y)=1+\langle P_0X, X\rangle\langle P_0Y, Y\rangle - \langle P_9X, Y\rangle^2=-1<0.$$

\noindent
$(8)$  \emph{\textbf{$M_+$ with $(m_1, m_2)=(8, 7)$ in the indefinite case}. }
On $\R^{32}$, we have a symmetric Clifford system $P_0,\ldots, P_9$ as that in (7).
Similarly as in the alternative proof for the indefinite $(4,3)$ case, removing $P_1$ from $P_0,\ldots, P_9$, we obtain a new symmetric Clifford system $Q_0=P_0, Q_{\alpha}=P_{1+\alpha}$ $(\alpha=1,\ldots,8)$ on $\R^{32}$. It is direct to see
$$Q_0\cdots Q_8(u,v)=(-Ju, Jv)=(-u_2, u_1, v_2, -v_1),$$
and thus $Q_0,\ldots, Q_8$ is an indefinite system.

Taking $x$, $X, Y$ the same with those in (7), we obtain the sectional curvature
$$K(X, Y)=1+\langle Q_0X, X\rangle\langle Q_0Y, Y\rangle - \langle Q_8X, Y\rangle^2=-1<0.$$

\vspace{1mm}

\noindent
\subsubsection{\textbf{$M_-$ of OT-FKM type.}}

As we mentioned in $(5)$ of last subsection, when $m_1=m=1$, \cite{TY13} showed that $M_-$ of OT-FKM type is isometric to $S^1(1)\times S^{l-1}(1)\big/\mathbb{Z}_2$. Thus the sectional curvature of $M_-$ in this case is non-negative.

We recall some basic properties of $M_-$. Given $x\in M_-$, there always exists $P$ in the unit sphere $\Sigma(P_0,\cdots, P_m)$ spanned by $P_0,\ldots, P_m$ such that $Px=x$. Denote $Q_0=P$, one can extend it to such a symmetric Clifford system $\{Q_0, \ldots, Q_m\}$ with $Q_i~ (i\geq 1)$ perpendicular to $Q_0$ and $\Sigma(Q_0,\cdots, Q_m)=\Sigma(P_0,\cdots, P_m)$. Choosing $\eta_1, \eta_2, \cdots, \eta_{l-m}$ as an orthonormal basis of $T^{\perp}_xM_-$ in $S^{2l-1}(1)$, Lemma 2.1 of \cite{TY15} reveals that for any $1\leq i\leq m$,
\begin{equation}\label{basis}
\{Q_i\eta_1,\cdots,Q_i\eta_{l-m},~~Q_1x,\cdots,Q_mx,~~Q_iQ_1x,\cdots,\widehat{Q_iQ_{i}x},\cdots,Q_iQ_mx\}
\end{equation}
constitute an orthonormal basis of $T_xM_-$. Moreover, we can decompose $A_{\alpha}X$ as
\begin{equation}\label{A}
A_{\alpha}X=\sum_{i=1}^m(\langle X, Q_ix\rangle Q_i\eta_{\alpha} + \langle X, Q_i\eta_{\alpha}\rangle Q_ix).
\end{equation}

When $m\geq 2$, we take $X=(Q_ix+Q_i\eta_{\alpha})\big/\sqrt{2}$, $Y=(Q_jx-Q_j\eta_{\alpha})\big/\sqrt{2}$ with $i, j>0, i\neq j$. A direct calculation by virtue of (\ref{A}) leads to $$\langle X, Y\rangle=0, ~~|X|=|Y|=1, ~~\langle A_{\beta} X, X\rangle=\delta_{\beta\alpha}, ~~\langle A_{\beta} Y, Y\rangle=-\delta_{\beta\alpha}.$$
Therefore, the sectional curvature
$$K(X, Y)=1+\sum_{\beta=1}^{l-m}\langle A_{\beta}X,
X\rangle\langle A_{\beta}Y, Y\rangle -\sum_{\beta=1}^{l-m} \langle A_{\beta}X,
Y\rangle^2=-\sum_{\beta=1}^{l-m} \langle A_{\beta}X,
Y\rangle^2.$$
Suppose $K(X, Y)\geq 0$. Then $\langle A_{\beta}X, Y\rangle=0$ for any $\beta=1,\ldots, l-m.$ Equivalently, $\langle A_{\beta}X, Q_jx\rangle=\langle A_{\beta}X, Q_j\eta_{\alpha}\rangle$ for any $j\neq i$. Furtherer, combining with (\ref{A}), we could derive that
$$\langle Q_iQ_j\eta_{\beta}, \eta_{\alpha}\rangle=\langle Q_jQ_i\eta_{\beta}, \eta_{\alpha}\rangle, ~~\forall~ i\neq j, ~~\forall~ \alpha, \beta,$$ which leads to
$$\langle Q_j\eta_{\beta},  Q_i\eta_{\alpha}\rangle=0, ~~\forall~ i\neq j, ~~\forall~ \alpha, \beta.$$
By conjunction with (\ref{basis}), this leads to $l-m-1\leq 0,$ which contradicts $m_2=l-m-1\geq 1.$

In conclusion, the sectional curvature of $M_-$ of OT-FKM type with $m\geq 2$ is not non-negative.

\vspace{3mm}

\noindent
\subsubsection{\textbf{ Focal submanifolds with $g=4$ and $(m_1, m_2)=(2,2)$.}}
According to \cite{QTY13}, one focal submanifold is diffeomorphic to the oriented Grassmann manifold $\widetilde{G_2}(\R^5)$, which is Einstein, and the other is diffeomorphic to $\C P^3$.

We first deal with the focal submanifold $M_-$ diffeomorphic to $\widetilde{G_2}(\R^5)=\frac{SO(5)}{SO(2)\times SO(3)}$.  As mentioned in Remark 4.1 of \cite{QTY13},  the induced metric on $M_-$ from the Euclidean space $\mathbb{R}^{10}$ is the unique invariant metric on the compact irreducible symmetric space
$\widetilde{G}_2(\mathbb{R}^5)$, because $M_-\subset \mathbb{R}^{10}$ is just the standard Pl{\"u}cker embedding of $\widetilde{G}_2(\mathbb{R}^5)$ into $\mathbb{R}^{10}$ (\cite{Sol92}). Thus the sectional curvature of $M_-$ diffeomorphic to $\widetilde{G_2}(\R^5)$ is non-negative.

As for the other focal submanifold $M_+$ diffeomorphic to $\C P^3$, we follow 4.1 (2) of \cite{QTY13}. Choosing a point $e'\in so(5, \mathbb{R})$ with coordinates $a_{12}=a_{34}=\frac{1}{\sqrt{2}}$ and zero otherwise, they gave the components of the second fundamental form of $M_+$ at $e'$ as follows
$$s_0=x_1^2+x_2^2-y_1^2-y_2^2,\quad
s_1=2(x_1y_1+x_2y_2),\quad
s_2=2(x_2y_1-x_1y_2),$$
where $\{x_1, x_2, y_1, y_2, z_1, z_2\}$ are the tangent coordinates. Polarize $s_0, s_1, s_2$ and take $X=(0,0,\frac{1}{\sqrt{2}}, \frac{1}{\sqrt{2}}, 0,0)$, $Y=(\frac{1}{\sqrt{2}}, \frac{1}{\sqrt{2}}, 0,0,0,0)$. It is easy to see that
\begin{equation*}
\left\{\begin{array}{ll}
\langle A_0X, ~X\rangle=-1, \,\,\, \langle A_0Y, Y\rangle=1,\\
\langle A_{\alpha}X, ~X\rangle=\langle A_{\alpha}Y, ~Y\rangle=0, \quad\alpha=1,2,\\
\langle A_1X, ~Y\rangle=1,\,\,\, \langle A_{\alpha}X, ~Y\rangle=0,\quad \alpha=0,2.
\end{array}
\right.
\end{equation*}
Therefore, the sectional curvature
$$K(X, Y)=1+\langle A_0X, X\rangle\langle A_0Y, Y\rangle -\langle A_1X, Y\rangle^2=-1<0.$$

\vspace{1mm}

\noindent
\subsubsection{\textbf{ Focal submanifolds with $g=4$ and $(m_1, m_2)=(4,5)$.}}
In this case, we also follow \cite{QTY13}, where they gave explicit components of the second fundamental form.

For the focal submanifold $M_+^{14}$, choosing a point $e\in so(5, \mathbb{C})$ with coordinates $a_{12}=a_{34}=\frac{1}{\sqrt{2}}$ and zero otherwise, the components of the second fundamental form of $M_+^{14}$ at $e$  are given by
\begin{eqnarray*}
&&s_0=x_{1}^2+\cdots +x_5^2-y_1^2-\cdots-y_5^2,\\
&&s_1=2(x_1y_1+\cdots+x_4y_4) +\sqrt{2}(x_5+y_5)z_1,\\
&&s_2=2(x_2y_1-x_1y_2)+2(x_3y_4-x_4y_3)+\sqrt{2}(x_5+y_5)z_2,\\
&&s_3=2(x_3y_1-x_1y_3)+2(x_4y_2-x_2y_4)+\sqrt{2}(x_5+y_5)z_3,\\
&&s_4=2(x_2y_3-x_3y_2)+2(x_4y_1-x_1y_4)+\sqrt{2}(x_5+y_5)z_4,
\end{eqnarray*}
where $(x_1,...,x_5,y_1,...,y_5,z_1,...,z_4)$ are the tangent coordinates.
Polarize $s_0, \ldots, s_4$ and take $X=(1,0,\ldots,0)$ with $x_1=1$ and $Y=(0,\ldots,0,1,0\ldots,0)$ with $y_1=1$. It is easy to see that
\begin{equation*}
\left\{\begin{array}{ll}
\langle A_0X, ~X\rangle=1, \,\,\, \langle A_0Y, Y\rangle=-1,\\
\langle A_{\alpha}X, ~X\rangle=\langle A_{\alpha}Y, ~Y\rangle=0, \quad\forall\alpha\neq 0,\\
\langle A_1X, ~Y\rangle=1,\,\,\, \langle A_{\alpha}X, ~Y\rangle=0,\quad \forall\alpha\neq 1
\end{array}
\right.
\end{equation*}
Therefore, the sectional curvature
$$K(X, Y)=1+\langle A_0X, X\rangle\langle A_0Y, Y\rangle -\langle A_1X, Y\rangle^2=-1<0.$$

For the focal submanifold $M_-^{13}$, choosing a point $e'\in so(5, \mathbb{C})$ with coordinates $a_{12}=-a_{21}=1$ and zero otherwise, the components of the second fundamental form of $M_-^{13}$ at $e'$  are given by
\begin{eqnarray*}
s_0&=& -2x_{14}x_{23}+2x_{13}x_{24}+2y_{14}y_{23}-2y_{13}y_{24},\\
s_1 &=& -2x_{15}x_{23}+2x_{13}x_{25}+2y_{15}y_{23}-2y_{13}y_{25}, \\
s_2 &=& -2x_{15}x_{24}+2x_{14}x_{25}+2y_{15}y_{24}-2y_{14}y_{25}, \\
s_3 &=& -2x_{14}x_{23}+2x_{13}y_{24}-2y_{14}x_{23}+2y_{13}x_{24}, \\
s_4 &=& -2x_{15}y_{23}+2x_{13}y_{25}-2y_{15}x_{23}+2y_{13}x_{25}, \\
s_5 &=& -2x_{15}y_{24}+2x_{14}y_{25}-2y_{15}x_{24}+2y_{14}x_{25}.
\end{eqnarray*}
where  $\{y_{12},\ x_{13},\ y_{13},$ $x_{14},\ y_{14},\ x_{15},\ y_{15},\ x_{23},$ $y_{23},\ x_{24},\ y_{24},\ x_{25},\ y_{25}\}$ are the tangent coordinates.
Polarize $s_0, \ldots, s_5$ and take $X$ with $x_{14}=x_{23}=\frac{1}{\sqrt{2}}$ and zero otherwise, $Y$ with $x_{13}=x_{24}=y_{24}=\frac{1}{\sqrt{3}}$ and zero otherwise. It is easy to see that
\begin{equation*}
\left\{\begin{array}{ll}
\langle A_0X, ~X\rangle=-1, ~~\langle A_3X, ~X\rangle=-1, ~~\langle A_{\alpha}X, ~X\rangle=0, ~~\forall \alpha\neq 0,3,\\
\langle A_0Y, Y\rangle=\frac{2}{3},~~\langle A_3Y, Y\rangle=\frac{2}{3}, ~~\langle A_{\alpha}Y, Y\rangle=0,~~ \forall \alpha\neq 0,3.
\end{array}
\right.
\end{equation*}
Therefore, the sectional curvature
$$K(X, Y)=1+\langle A_0X, X\rangle\langle A_0Y, Y\rangle +\langle A_3X, X\rangle\langle A_3Y, Y\rangle-\sum_{\alpha=0}^5\langle A_{\alpha}X, Y\rangle^2\leq -\frac{1}{3}<0.$$

\noindent
\subsubsection{\textbf{ Focal submanifolds with $g=6$ and $(m_1, m_2)=(1, 1)$.}}

Given $p\in M_+^{5} \subset S^7$, with respect to a suitable tangent orthonormal basis $e_1, \ldots, e_5$ of $T_pM_+$, Miyaoka \cite{Miy93} showed that the shape operators of $M_+$ are given by
\begin{equation}\label{M+5}
A_0=  \left(
\begin{smallmatrix}
 \sqrt{3}& & & & \\
 & \frac{1}{\sqrt{3}}&  & &\\
& & 0 &&\\
  &&  & -\frac{1}{\sqrt{3}}&\\
  &&&&-\sqrt{3}
\end{smallmatrix}\right),\quad
 A_1=  \left(\begin{smallmatrix}
 & & &   &  \sqrt{3} \\
 & & & -\frac{1}{\sqrt{3}} & \\
 & & 0&  & \\
&-\frac{1}{\sqrt{3}} &  & &  \\
\sqrt{3}  & & & &
\end{smallmatrix}\right).
\end{equation}
A direct calculation leads to
$$K(e_1, e_5)=1+\sum_{\alpha=0}^1\langle A_{\alpha}e_1, e_1\rangle\langle A_{\alpha}e_5, e_5\rangle-\sum_{\alpha=0}^1\langle A_{\alpha}e_1, e_5\rangle^2=-5<0.$$

Similarly, for the focal submanifold $M_-^5$, the shape operators of $M_-$ are given by
\begin{equation}\label{M-5}
A_0=  \left(
\begin{smallmatrix}
 \sqrt{3}& & & & \\
 & \frac{1}{\sqrt{3}}&  & &\\
& & 0 &&\\
  &&  & -\frac{1}{\sqrt{3}}&\\
  &&&&-\sqrt{3}
\end{smallmatrix}\right),\quad
 A_1=  \left(\begin{smallmatrix}
0 &-1 & 0 & 0  &  0 \\
-1 &0 &0 & \frac{2}{\sqrt{3}} & 0\\
0 &0 & 0& 0 &0 \\
0&\frac{2}{\sqrt{3}} & 0 &0 &-1  \\
0 & 0& 0&-1 &0
\end{smallmatrix}\right).
\end{equation}
A direct calculation leads to
$$K(e_1, e_5)=1+\sum_{\alpha=0}^1\langle A_{\alpha}e_1, e_1\rangle\langle A_{\alpha}e_5, e_5\rangle-\sum_{\alpha=0}^1\langle A_{\alpha}e_1, e_5\rangle^2=-2<0.$$

\noindent
\subsubsection{\textbf{ Focal submanifolds with $g=6$ and $(m_1, m_2)=(2,2)$.}}
Given $p\in M_+^{10} \subset S^{13}$, with respect to a suitable tangent orthonormal basis $e_1, \ldots, $ of $T_pM_+$, Miyaoka \cite{Miy13} showed that the shape operators of $M_+$ are given by
\begin{equation}\label{M+10}
A_0=  \left(
\begin{smallmatrix}
 \sqrt{3}I& & & & \\
 & \frac{1}{\sqrt{3}}I&  & &\\
& & 0 &&\\
  &&  & -\frac{1}{\sqrt{3}}I&\\
  &&&&-\sqrt{3}I
\end{smallmatrix}\right),\quad
A_1=  \left(
\begin{smallmatrix}
 & & & & \sqrt{3}J\\
 &&  &  \frac{1}{\sqrt{3}}J&\\
& & 0 &&\\
  & -\frac{1}{\sqrt{3}}J&  &&\\
 -\sqrt{3}J &&&&
\end{smallmatrix}\right)\end{equation}
and
$$
A_2=  \left(
\begin{smallmatrix}
& & & &  \sqrt{3}I\\
 & &  & \frac{1}{\sqrt{3}}I&\\
& & 0 &&\\
  &\frac{1}{\sqrt{3}}I&  & &\\
\sqrt{3}I  &&&&
\end{smallmatrix}\right), ~~\rm{where}~~~I=\left(
\begin{smallmatrix}
1&0\\
0&1
\end{smallmatrix}\right), J=\left(
\begin{smallmatrix}
0&-1\\
-1&0
\end{smallmatrix}\right)$$
A direct calculation leads to
$$K(e_1, e_{10})=1+\sum_{\alpha=0}^2\langle A_{\alpha}e_1, e_1\rangle\langle A_{\alpha}e_{10}, e_{10}\rangle-\sum_{\alpha=0}^2\langle A_{\alpha}e_1, e_{10}\rangle^2=-5<0.$$

Similarly, for the focal submanifold $M_-^{10}$, the shape operators of $M_-$ are given by
\begin{equation}\label{M-10}
A_0=  \left(
\begin{smallmatrix}
 \sqrt{3}I& & & & \\
 & \frac{1}{\sqrt{3}}I&  & &\\
& & 0 &&\\
  &&  & -\frac{1}{\sqrt{3}}I&\\
  &&&&-\sqrt{3}I
\end{smallmatrix}\right),\quad
 A_1=  \left(\begin{smallmatrix}
0 &-I & 0 & 0  &  0 \\
-I &0 &0 & \frac{2}{\sqrt{3}}I & 0\\
0 &0 & 0& 0 &0 \\
0&\frac{2}{\sqrt{3}}I & 0 &0 &-I  \\
0 & 0& 0&-I &0
\end{smallmatrix}\right).\end{equation}
and
$$ A_2=  \left(\begin{smallmatrix}
0 &J & 0 & 0  &  0 \\
-J &0 &0 & -\frac{2}{\sqrt{3}}J & 0\\
0 &0 & 0& 0 &0 \\
0&\frac{2}{\sqrt{3}}J & 0 &0 & J \\
0 & 0& 0&-J &0
\end{smallmatrix}\right)$$
A direct calculation leads to
$$K(e_1, e_{10})=1+\sum_{\alpha=0}^2\langle A_{\alpha}e_1, e_1\rangle\langle A_{\alpha}e_{10}, e_{10}\rangle-\sum_{\alpha=0}^2\langle A_{\alpha}e_1, e_{10}\rangle^2=-2<0.$$

\end{proof}

\vspace{2mm}

\subsection{Ricci curvature.}

As we introduced before, the Ricci curvature of an isoparametric hypersurface $M^n$ in $S^{n+1}(1)$ with $g=1$ is obviously positive and that with $g=2$ is positive unless it is $S^1(r_1)\times S^{n-1}(r_2)$, where the Ricci curvature could be zero.

For other cases, we derive the following proposition on Ricci curvature of an isoparametric hypersurface

\begin{prop}\label{Ricci of hyp}
For an isoparametric hypersurface $M^n$ in $S^{n+1}(1)$, we have
\begin{itemize}
\item[(i)] When $g=3, m=1$, the Ricci curvature is not non-negative;
when $g=3, m>1$, the Ricci curvature is positive if $M^n$ is close to the minimal isoparametric hypersurface;
\item[(ii)] When $g=4, m_1=1$ or $m_2=1$, the Ricci curvature is not non-negative;
when $g=4, m_1, m_2\geq 2$ the Ricci curvature is positive if $M^n$ is close to the minimal isoparametric hypersurface;
\item[(iii)] When $g=6, m=1$, the Ricci curvature is not non-negative;
when  $g=6, m=2$, the Ricci curvature is not non-negative if $M^n$ is close to the minimal isoparametric hypersurface.
\end{itemize}
\end{prop}

\begin{proof}
Denote by $e_1,\ldots, e_n$ an orthonormal basis of $T_pM^n$ corresponding to principal curvatures $\lambda_1\geq \lambda_2\geq\cdots\geq\lambda_n.$
It follows from Gauss equation that
$Ric(e_i)=n-1+\lambda_i H-\lambda_i^2 ,$
where $H=\lambda_1+\cdots+\lambda_n$ is the mean curvature of $M^n$.

(i) When $g=3, m=1$, denote $\lambda_1=\cot\theta$ with $\theta\in(0, \frac{\pi}{3})$, then $H=3\cot 3\theta$, and $$Ric(e_1)=2+3\cot\theta\cot3\theta-\cot^2\theta=-\frac{2}{3}-\frac{8}{3(3\cot^2\theta-1)}<0.$$

When $g=3, m\geq 2$, we need only to prove the positivity of the Ricci curvature for the minimal isoparametric hypersurface, since the principal curvatures are continuous functions on $M^n$. Let $X=\sum_{i=1}^na_ie_i$ with $\sum_{i=1}^na_i^2=1$ be a unit tangent vector, then $Ric(X)=n-1+\sum_i\lambda_ia_i^2H-\sum_i\lambda_i^2a_i^2.$
In the minimal case, $\theta=\frac{\pi}{6}$ and
\begin{eqnarray*}
Ric(X)&=&n-1-\sum_i\lambda_i^2a_i^2\geq n-1- \max\{\cot^2\theta, \cot^2(\theta+\frac{2}{3}\pi)\}\\
&\geq&n-4=3m-4>0.
\end{eqnarray*}

(ii) When $g=4$, denote $\lambda_1=\cot\theta$ with $\theta\in (0, \frac{\pi}{4})$, then the four distinct principal curvatures can be expressed as $\lambda_1, \frac{\lambda_1-1}{\lambda_1+1}, -\frac{1}{\lambda_1}, -\frac{\lambda_1+1}{\lambda_1-1}$. Thus
\begin{equation}\label{H}
H=m_1\frac{\lambda_1^2-1}{\lambda_1}-4m_2\frac{\lambda_1}{\lambda_1^2-1}.
\end{equation}

In case $m_1=1$, a direct calculation leads to 
$$Ric(e_1)=2(1+m_2)-1+\lambda_1(H-\lambda_1)=-2m_2\frac{\lambda_1^2+1}{\lambda_1^2-1}<0.$$ 
The discussion for case $m_2=1$ is similar.

In case $m_1, m_2\geq 2$, we only consider Ricci curvature of the minimal isoparametric hypersurface. By (\ref{H}), $H=0$ implies that $$\lambda_1=\cot\theta=\sqrt{\frac{m_2}{m_1}}+\sqrt{\frac{m_2}{m_1}+1},$$ 
and thus
$\lambda_1^2<2m_2+3\leq2(m_1+m_2)-1=n-1.$ Similarly, $\lambda_n^2=\cot^2(\theta+\frac{3}{4}\pi)<2m_1+3\leq n-1.$ Therefore,
$\min Ric(X)=n-1-\max\{\lambda_1^2, \lambda_n^2\}>0.$

(iii) When $g=6, m=1$, denote $\lambda_1=\cot\theta$ with $\theta\in(0, \frac{\pi}{6})$, then $H=6\cot6\theta$. It is direct to compute that $$Ric(e_1)=n-1+\lambda_1(H-\lambda_1)=-4\frac{(\lambda_1^2+1)(5\lambda_1^2-3)}{(\lambda_1^2-3)(3\lambda_1^2-1)}<0,$$ 
since $\lambda_1>\sqrt{3}$.

 When $g=6, m=2$ and $M^n$ is minimal, $H=0$ implies that $\theta=\frac{\pi}{12}$, and $\lambda_1=\cot\theta=2+\sqrt{3}=-\lambda_n$. Therefore,
$\min Ric(X)=n-1-\max\{\lambda_1^2, \lambda_n^2\}=11-(2+\sqrt{3})^2<0.$
\end{proof}

\begin{rem}
In Lemma 2 of \cite{Wu94}, the positivity of the Ricci curvature of a minimal isoparametric hypersurface in the case $g=4, m_1, m_2\geq 2$ is also discussed.
\end{rem}

Next, we consider the Ricci curvature of focal submanifolds. As we discussed in the beginning of last subsection, it is obvious that the Ricci curvatures of the focal submanifolds are positive when $g=2, 3$. When $g=4$, it was dealt in (5.2) of \cite{TY15} that the Ricci curvature of $M_+$ satisfies $Ric(X)\geq 2(m_2-1)$. Thus it is positive if $m_2>1$. Similarly,  the Ricci curvature of $M_-$ is positive if $m_1>1$. We deal with the cases with $g=6$ and obtain the following theorem

\begin{prop}\label{Ricci of focal}
For focal submanifolds of an isoparametric hypersurface with $g=6$ in $S^{n+1}(1)$, we have
\begin{itemize}
\item[(i)] When $m=1$, the Ricci curvature is not non-negative;
\item[(ii)] When $m=2$, the Ricci curvature is non-negative.
\end{itemize}
\end{prop}

\begin{proof}
(i). For $M_+^5$, let $e_1,\ldots, e_5$ be orthonormal basis of $T_pM_+$ corresponding to (\ref{M+5}). Let $X=\sum_{i=1}^5a_ie_i$ with $\sum_{i=1}^5a_i^2=1$.
Combining with the minimality of focal submanifolds in the unit sphere, we can calculate the Ricci curvature directly
\begin{equation}\label{Ricci of M+5}
Ric(X)=(n-1)|X|^2-\sum_{\alpha=0}^1|A_{\alpha}X|^2=-2a_1^2+\frac{10}{3}a_2^2+4a_3^2+\frac{10}{3}a_4^2-2a_5^2.
\end{equation}
Clearly, it is not non-negative.

For $M_-^5$,  let $e_1,\ldots, e_5$ be orthonormal basis of $T_pM_-$ corresponding to (\ref{M-5}). Let $X=\sum_{i=1}^5a_ie_i$ with $\sum_{i=1}^5a_i^2=1$.
A direct calculation leads to
\begin{equation}\label{Ricci of M-5}
Ric(X)=(n-1)|X|^2-\sum_{\alpha=0}^1|A_{\alpha}X|^2=\frac{4}{3}a_2^2+4a_3^2+\frac{4}{3}a_4^2+\frac{4}{\sqrt{3}}a_1a_4+\frac{4}{\sqrt{3}}a_2a_5.
\end{equation}
Again, it is not non-negative. For example, take $a_1=-\frac{\sqrt{2}}{2}, a_4=\frac{\sqrt{2}}{2}$, $a_1=a_3=a_5=0$, then $Ric(X)=\frac{2}{3}-\frac{2}{\sqrt{3}}<0.$

Moreover, as a quadratic form, the eigenvalues of $Ric(X)$ in (\ref{Ricci of M+5}) are $-2, -2, \frac{10}{3}, \frac{10}{3}, 4$, and that in (\ref{Ricci of M-5}) are $-\frac{2}{3}, -\frac{2}{3}, 2, 2, 4$. Therefore, the intrinsic geometry of the two focal submanifolds with $g=6, m=1$ are essentially different. Especially, they are not isometric to each other as mentioned in \cite{TXY14}.
\vspace{2mm}

(ii). For $M_+^{10}$, using (\ref{M+10}), we take a similar process as in (i) and obtain 
\begin{equation}\label{Ricci of M+10}
Ric(X)=(n-1)|X|^2-\sum_{\alpha=0}^2|A_{\alpha}X|^2=8(a_3^2+a_4^2+a_7^2+a_8^2)+9(a_5^2+a_6^2),
\end{equation}
which is obviously non-negative.

For $M_-^{10}$, using (\ref{M-10}), a similar process as in (i) leads to
\begin{equation}\label{Ricci of M-10}
Ric(X)=(n-1)|X|^2-\sum_{\alpha=0}^2|A_{\alpha}X|^2=4\sum_{i=1}^{10}a_i^2+5(a_5^2+a_6^2),
\end{equation}
which is obviously non-negative.

From (\ref{Ricci of M+10}), (\ref{Ricci of M-10}), we can see directly that the eigenvalues of the Ricci curvature of $M_+^{10}$ and $M_-^{10}$ are different, thus they are not isometric.
\end{proof}

\end{document}